\documentclass[reqno,a4paper]{amsart}

\usepackage{txfonts,amsmath,amssymb,bbm,amsthm,mathrsfs,enumitem,hyperref,color,comment,needspace}
\hypersetup{hidelinks,hypertexnames=false}
\numberwithin{equation}{section}

\newtheorem{theorem}{Theorem}[section]
\newtheorem{proposition}[theorem]{Proposition}
\newtheorem{lemma}[theorem]{Lemma}

\newtheorem{assumption}{Assumption}
\newtheorem{definition}{Definition}[section]

\newtheorem{corollary}{Corollary}[section]

\newcommand{\X}{\mathsf X}
\newcommand{\Y}{\mathsf Y}
\newcommand{\Z}{\mathsf Z}
\newcommand{\Fc}{\mathcal F}
\newcommand{\Pc}{\mathcal P}
\newcommand{\Fb}{\mathbb F}
\newcommand{\Prob}{\mathbb P}
\newcommand{\PathLaw}[1]{\mathbf{#1}}
\newcommand{\E}{\mathbb E}

\newcommand{\Ent}{H}
\newcommand{\FI}{I}
\newcommand{\PFI}{\mathcal I}

\newcommand{\Wass}[1]{W_{#1}}
\newcommand{\TV}{d_{\mathrm{TV}}}
\newcommand{\Hell}{d_{\mathrm H}}

\newcommand{\Drift}{\mathsf b}

\newcommand{\CondScore}[2]{m_{#1\mid#2}}

\newcommand{\Bridge}[2]{\bar\pi^{#1,#2}}

\DeclareMathOperator*{\argmin}{arg\,min}

\title{Initialization Stability and Convergence of Projected Langevin Dynamics}
\author{Qi Feng}
\address{Department of Mathematics, Florida State University}
\email{qfeng2@fsu.edu}

\author{Nathan Sauldubois}
\address{Department of Finance and Risk Engineering, New York University}
\email{ns6982@nyu.edu}

\author{Xin Zhang}
\address{Department of Finance and Risk Engineering, New York University}
\email{xz1662@nyu.edu}
\thanks{Q. Feng
 is partially supported by the National Science Foundation under grant \#DMS-2420029.
X. Zhang is partially supported by the NSF Grant DMS-2508556. N. Sauldubois and X. Zhang thank Daniel Lacker for helpful discussions.}

\date{}

\begin{document}

\begin{abstract}
The entropic regularized optimal transport (EOT) plays a fundamental role in the computation of optimal transport. In \cite{CLP}, the authors proposed a gradient flow approach to EOT, and derived a projected Langevin diffusion that converges in distribution to the solution of EOT when initialized with a coupling between the two given marginals.
In this paper, motivated by samplings of optimal couplings for EOT, we analyze the projected Langevin dynamics with arbitrary initial distribution, and prove that its long time behavior is the same as if it was initialized from a coupling. 
\end{abstract}

\maketitle

\section{Introduction}
Let $\X=\mathbb T^{d_x}$, $\Y=\mathbb T^{d_y}$ be flat tori of dimensional $d_x$ and $d_y$ respectively,  and
$\Z=\X\times\Y$. Given a lower semi-continuous cost function $c: \X \times \Y \to \mathbb R$ and marginal distributions $\mu \in \mathcal{P}(\X)$, $\nu \in \mathcal{P}(\Y)$, the entropic regularized optimal transport (EOT) seeks to solve the optimization problem 
\vspace{-1mm}
$$
    \inf_{P \in \Pi(\mu,\nu)}J(P), \qquad  J(P):=\int_{\Z} c \, dP+\epsilon H(P \mid \mu \otimes \nu), 
    \vspace{-1mm}
$$
where $\Pi(\mu,\nu)$ is the set of couplings of $\mu$ and $\nu$ and $H(P\mid\mu\otimes\nu)$ is the relative entropy of $P$ with respect to $\mu\otimes\nu$. This problem admits a unique minimizer $\pi=\argmin_{P\in\Pi(\mu,\nu)}J(P)$, and has been extensively studied for its central role in computational optimal transport, notably through the Sinkhorn algorithm.

In \cite{CLP}, the authors proposed a gradient flow approach to EOT, viewing $\Pi(\mu,\nu)$ as a Riemannian submanifold of $(\mathcal{P}(\Z),W_2)$. 
% , they considered continuity equation
% \begin{align*}
%     \partial_t P_t=\operatorname{div}( \nabla_{W}^{\Pi} J(P_t)P_t),
% \end{align*}
% where $\nabla_W^{\Pi}$ denotes the Wasserstein gradient on $\Pi(\mu,\nu)$ in the sense of sub-Riemannian.
Suppose $\mu(dx)=e^{-U(x)}\,dx$ and $\nu(dy)=e^{-V(y)}\,dy$. Under suitable assumptions on $U,V,c$, \cite{CLP} formally derived a projected Langevin dynamics associated with the gradient flow
\begin{equation}\label{eq:dynamics}
\left\{
\begin{aligned}
 dX_t={}&\bigl(\E^{P_t}[\nabla_xc(X_t,Y_t)\vert X_t]
 -\nabla_xc(X_t,Y_t)-\epsilon\nabla U(X_t)\bigr) \, dt
 +\sqrt{2\epsilon}\,dW_t,\\
 dY_t={}&\bigl(\E^{P_t}[\nabla_yc(X_t,Y_t)\vert Y_t]
 -\nabla_yc(X_t,Y_t)-\epsilon\nabla V(Y_t)\bigr) \, dt
 +\sqrt{2\epsilon}\,dB_t,
\end{aligned}
\right.  \tag{PLD}
\end{equation}
where $P_t:=\mathcal{L}(X_t,Y_t)$ denotes the joint law of $X_t,Y_t$, and processes $(W_t)$, $(B_t)$ are independent Brownian motions of dimensions $d_x$
and $d_y$, respectively.

In \cite{CLP}, the authors established the well-posedness of \eqref{eq:dynamics} and the convergence $P_t \to \pi$ as $t \to \infty$ for $P_0 \in \Pi(\mu,\nu)$, leaving the case $P_0 \notin \Pi(\mu,\nu)$ open. Furthermore in data-driven applications, however, the marginals $\mu,\nu$, or equivalently the potentials $U,V$, are generally unknown, while the gradients $\nabla U,\nabla V$ can be estimated directly from data via score matching \cite{Hyvarinen2005,SongErmon2019}. Motivated by applications to sampling of optimal couplings for EOT, we therefore analyze the system \eqref{eq:dynamics} when only gradients $\nabla U, \nabla V$ are available. 

Here is our standing assumption and main result. Let $W_2$ denote the $2$-Wasserstein distance on $\mathcal{P}(\Z)$, and $\pi$ be the minimizer of EOT.

\begin{assumption}\label{ass:regularity}
There exists an $\alpha_0 \in (0,1)$ such that 
\[
 U\in C^{1,\alpha_0}(\X),\qquad
 V\in C^{1,\alpha_0}(\Y),\qquad
 c\in C^{1,\alpha_0}(\Z).
\]
\end{assumption}
\begin{theorem}
\label{thm:main}
Under Assumption~\ref{ass:regularity}, \eqref{eq:dynamics} admits a unique weak
solution for every $P_0\in\Pc(\Z)$. Moreover, there exist $r>0$ and $C_\tau<\infty$ for each $\tau>0$, such that every weak solution to \eqref{eq:dynamics}  satisfies
\begin{equation}\label{eq:main-wasserstein-bound}
\Ent(P_t\mid\pi)
\le C_\tau e^{-r(t-\tau)}\Wass{2}(P_0,\pi)^2,
\qquad t\ge\tau. \tag{Rate}
\end{equation}
Here $r$ depends only on $\X,\Y,c,U,V,\epsilon$, while $C_\tau$ additionally depends on $\tau$.
\end{theorem}

The restriction to $t\ge\tau$ in~\eqref{eq:main-wasserstein-bound} is essential for singular initializations: if $P_0\not\ll\pi$, weak convergence to $P_0$ and lower semicontinuity of entropy give $\liminf_{t\downarrow0}\Ent(P_t\mid\pi)=+\infty$. We show that the uniform ellipticity of~\eqref{eq:dynamics} smooths the density at every positive time, yielding a finite entropy bound in terms of the initial Wasserstein error.
To explain the main idea behind~\eqref{eq:main-wasserstein-bound}: for $P,Q\in\mathcal{P}(\Z)$ with positive densities and $R:=\log\frac{dP}{dQ}$, the projected Fisher information is
\vspace{-1.5mm}
\begin{equation}\label{eq:projected-fisher}
 \PFI(P\mid Q)
 :=
 \E^P 
 \big[
 \vert\nabla_x R 
 -\E^{P} [\nabla_x R \mid X ] 
 \vert^2 
+
 \vert \nabla_y R 
 -\E^{P}  [ \nabla_y R \mid Y ]
 \vert^2
 \big]. \tag{PFI}
 \vspace{-1.5mm}
 \end{equation}
In \cite{CLP}, the convergence rate is based on the \emph{projected log-Sobolev inequality},
\begin{align*}
    H(P \mid \pi) \leq C \PFI(P \mid \pi), \quad \forall \, P \in \Pi (\mu,\nu), 
\end{align*}
where $C$ is a positive constant depending on $\pi$. However if $P_0 \not \in \Pi(\mu,\nu)$, $P_t \not \in \Pi(\mu,\nu)$ for $t \geq 0$. Hence the projected log-Sobolev inequality cannot be applied directly.

The main ingredient is the construction of auxiliary martingale bridges 
\vspace{-1mm}
\begin{align*}
    \bar \pi^{\mu_t,\nu_t}:=\argmin_{Q \in \Pi(\mu_t,\nu_t)} H(Q \mid \pi). 
    \vspace{-1mm}
\end{align*}
According to the Pythagorean theorem of relative entropy, 
\vspace{-1mm}
\begin{align*}
    H(P_t \mid \pi)=H(P_t \mid \bar \pi^{\mu_t,\nu_t})+H(\bar \pi^{\mu_t,\nu_t} \mid \pi).
    \vspace{-1mm}
\end{align*}
The convergence rate of $H(P_t \mid \bar \pi^{\mu_t,\nu_t})$ will be proved based on an energy dissipation equality in Lemma~\ref{lem:bridge-entropy-evolution} and projected log-Sobolev inequality in Proposition~\ref{prop:global-plsi}. The second term on the right $H(\bar \pi^{\mu_t,\nu_t} \mid \pi)$ is bounded from above by $H(\mu_t \mid \mu)+H(\nu_t \mid \nu)$. These two estimates yield the exponential convergence rate.

\subsection{Related literature}

EOT is closely related to the Schr\"odinger problem; see~\cite{Leonard2014,Nutz2022} for surveys. Its optimizer has the form
\begin{equation}\label{eq:pi-form}
\pi(dx,dy)=e^{\phi(x)+\psi(y)-c(x,y)/\epsilon}\,\mu(dx)\nu(dy)
\end{equation}
for some potentials $\phi,\psi$, a structure that underlies Sinkhorn's algorithm. Cuturi~\cite{Cuturi2013} showed the computational advantages of entropic regularization for discrete optimal transport; see~\cite{PeyreCuturi2019} for a broad treatment of computational optimal transport, and~\cite{BenamouEtAl2015} for iterative Bregman projection methods based on a Kullback--Leibler formulation. Stability of EOT with respect to the marginals and the cost has also been studied extensively, e.g.~\cite{GhosalNutzBernton2022,EcksteinNutz2022,NutzWiesel2023}; such static results are relevant to our construction of entropy-minimizing bridges along the evolving marginals of~\eqref{eq:dynamics} in Section~\ref{sec:local-static-estimates}.

Langevin diffusions provide a probabilistic realization of entropy
gradient flows. The law of the overdamped Langevin diffusion solves a
Fokker--Planck equation, whose interpretation as a Wasserstein gradient
flow of free energy is developed by Jordan, Kinderlehrer, and
Otto~\cite{JordanKinderlehrerOtto1998}. Relative entropy dissipates through
relative Fisher information, and a logarithmic Sobolev inequality yields
exponential convergence to equilibrium; see~\cite{BakryGentilLedoux2014}.
From the sampling perspective, Dalalyan~\cite{Dalalyan2017} and Durmus and
Moulines~\cite{DurmusMoulines2017} establish nonasymptotic convergence
bounds for discretizations of overdamped Langevin dynamics under
appropriate assumptions on the target density.

The closest work  of our paper is that of Conforti, Lacker, and Pal~\cite{CLP}, who
introduce projected Langevin dynamics and establish weak well-posedness
and convergence to the entropic optimal coupling for every initial law
in $\Pi(\mu,\nu)$, under suitable regularity and integrability assumptions
in the Euclidean setting.  For quadratic cost, they also
obtain exponential convergence through a projected logarithmic Sobolev
inequality.

Sauldubois and Zhang~\cite{SaulduboisZhang2026} consider a related extension to entropy-regularized weak optimal transport, where the cost may depend nonlinearly on the conditional laws of the coupling. Projecting onto couplings with prescribed marginals, they formally derive a gradient flow in adapted Wasserstein space and, under suitable assumptions, prove weak well-posedness and convergence to the entropic weak transport optimizer in the adapted Wasserstein topology.

Our argument is also related to entropy-dissipation methods for nonlinear diffusions: Bayraktar, Feng and Li~\cite{BFL} obtain exponential convergence for Vlasov--Fokker--Planck equations with weak self-consistency via an auxiliary Fisher information functional. The separation of conditional and marginal contributions is close in spirit to the two-scale method of~\cite{GOVW2009}, which combines functional inequalities for conditional measures and for a coarse-grained measure.
\vspace{-1mm}
% For~\eqref{eq:dynamics}, the projected Fisher information controls the
% coupling error within a set of laws with fixed marginals, while the
% marginal discrepancy contributes an additional term to the entropy
% balance. We handle these two errors by introducing the
% entropy-minimizing coupling $\Bridge{\mu_t}{\nu_t}$ relative to $\pi$.
% Its multiplicative form yields an exact entropy decomposition and leaves
% the projected Fisher information unchanged. Local estimates for the
% scaling potentials and bridge entropy, together with a projected
% logarithmic Sobolev inequality uniform over suitable reference densities,
% provide the required static control. Positive-time regularization and
% compactness then combine with the autonomous marginal entropy
% dissipation to give~\eqref{eq:main-wasserstein-bound} by a direct Gronwall
% argument.

\subsection{Organization of the paper}
Section~\ref{sec:preliminaries} establishes weak existence and uniqueness of \eqref{eq:dynamics},  and regularity of density for positive time.  It concludes
with the Wasserstein-to-entropy estimate.
Section~\ref{sec:local-static-estimates}  proves a local
projected logarithmic Sobolev inequality.
Section~\ref{sec:convergence-rate} derives the marginal and joint entropy
identities and 
completes the proof of Theorem~\ref{thm:main}.

\Needspace{7\baselineskip}
\subsection{Notation}
\label{sec:notation}
The spaces $\X,\Y,\Z$ carry their flat metrics and volume measures $dx$, $dy$ and $dz=dx\,dy$, and $d_\Z$ is the product geodesic distance. For $z=(x,y)$, $\nabla=(\nabla_x,\nabla_y)$, $\Delta=\Delta_x+\Delta_y$ and $\operatorname{div}$ act in space; a vector field $v$ has components $v_\X,v_\Y$ and Euclidean norm $|v|$.

For a measurable space $E$, $\Pc(E)$ is the set of probability measures on $E$. We use $P,Q$ for spatial laws and $p,q$ for their densities with respect to the volume measure, when they exist; in spatial derivatives and function norms, a measure may be identified with its density. We write $\E^{P}[f]=\int f\,dP$, and $\Pi(\alpha,\beta)$ for the set of couplings of $\alpha,\beta$, i.e.\ probability measures on the product space with marginals $\alpha,\beta$. The relative entropy of $\alpha$ with respect to $\beta$ is
\begin{equation}\label{eq:relative-entropy}
H(\alpha\mid\beta)=\int\log\tfrac{d\alpha}{d\beta}\,d\alpha\ \text{ if }\alpha\ll\beta,\qquad H(\alpha\mid\beta)=+\infty\ \text{ otherwise}.
\end{equation}

On $\Z$, the coordinate maps are $X(x,y)=x$ and $Y(x,y)=y$. The marginals and conditional laws satisfy
\vspace{-1.5mm}
\begin{equation}
 P_\X =P\circ X^{-1},\qquad P_\Y=P\circ Y^{-1},
 \quad \text{and}  \quad
 P(dx,dy) =P_\X(dx)P_{\Y\mid x}(dy).
 \label{eq:conditional-laws}
 \vspace{-1mm}
\end{equation}
Here $P_{\Y\mid x}$ is a Borel probability kernel, and similarly for $P_{\X\mid y}$. Conditional expectations are taken in jointly Borel versions, defined up to marginal null sets; see \cite[Proposition~5.1]{BrunickShreve2013} or \cite[Proposition~3.3]{CLP}.

For $E\in\{\X,\Y,\Z\}$ and $0<\theta<1$, $C^{1,\theta}(E)$ is the space of $C^1$ functions with $\theta$-H\"older gradient, normed by $\|f\|_{C^{1,\theta}}=\|f\|_\infty+\|\nabla f\|_\infty+[\nabla f]_\theta$, where $[g]_\theta=\sup_{z\ne z'}|g(z)-g(z')|/d_E(z,z')^\theta$; $C^{(1+\theta)/2,1+\theta}$ is the corresponding parabolic H\"older space. The same notation is used for vector-valued functions, and domains are omitted when clear from context. $L^p(Q)$ and $H^1(Q)$ carry the norms $(\int|f|^p\,dQ)^{1/p}$ (essential supremum for $p=\infty$) and $(\int(|f|^2+|\nabla f|^2)\,dQ)^{1/2}$; plain $L^p$, $H^1$ refer to the volume measure. When $Q$ has a positive continuous density, compactness makes $H^1(Q)$ and $H^1$ equal as sets, with equivalent norms.

For $P,Q\in\Pc(\Z)$ and $p=1,2$, the $p$-Wasserstein distance is defined by
\begin{equation}
 \Wass{p}(P,Q)^p=\inf_{\gamma\in\Pi(P,Q)}\int_{\Z\times\Z}d_\Z(z,z')^p\,d\gamma(z,z').
 \label{eq:wasserstein-distances}
\end{equation}
Here the infimum runs over all transport plans with marginals $P,Q$. These distances
are finite for all probability measures on the compact space $\Z$, so no moment restriction
is needed.

% Relative entropy $\Ent(P\mid Q)$ is defined in~\eqref{eq:relative-entropy}, and $\PFI(P\mid Q)$ denotes the
% projected Fisher information in~\eqref{eq:projected-fisher}.
% All logarithmic gradients below are understood as weak gradients.

For a weak solution of~\eqref{eq:dynamics}, let $P_t=\mathcal{L}(X_t,Y_t)$, $p_t$ its density for $t>0$, and $\mu_t=(P_t)_\X$, $\nu_t=(P_t)_\Y$. For $P\in\mathcal{P}(\Z)$, the drift is $\Drift[P]=(\Drift_\X[P],\Drift_\Y[P])$, where
\begin{equation}\label{eq:drift-definition}
\Drift_\X[P](x,y):=\E^{P}[\nabla_xc\vert X=x]-\nabla_xc(x,y)-\epsilon\nabla U(x),
\end{equation}
and $\Drift_\Y[P]$ is defined symmetrically, with $\nabla_yc$, $Y=y$ and $\nabla V(y)$.

\section{Well-posedness and regularization}
\label{sec:preliminaries}

In this section, we establish several preliminary results that will be needed for the analysis of stability with respect to the initial condition. We first prove the well-posedness of equation \eqref{eq:dynamics}, in the sense of \cite[Definition 3.1]{CLP}. We recall the definition for completeness.

\begin{definition}
Let $P_0 \in \Pc(\X\times\Y)$. A weak solution of equation \eqref{eq:dynamics} consists of a filtered probability space
$
(\Omega,\Fc,\Fb, \Prob),
$
supporting independent Brownian motions $W$ and $B$ of dimensions $d_x$ and $d_y$, respectively, together with two $\Fb$-adapted continuous processes $X$ and $Y$ with values in $\X$ and $\Y$, respectively, and two $\Fb$-progressively measurable processes $\hat A$ and $\hat B$ with values in $\mathbb R^{d_x}$ and $\mathbb R^{d_y}$, respectively, such that:
\begin{enumerate}[label=(\roman*), ref=(\roman*)]

\item
\(
(X_0,Y_0)\sim P_0
\).

\item \label{weak-sol:MV-a:int}
There exists $p>1$ such that, for every $T>0$,
\[
\int_0^T
\E^{\Prob}\left[
\left|
\hat A_t
-
\nabla_x c(X_t,Y_t)
-
\epsilon \nabla_x U(X_t)
\right|^p
+
\left|
\hat B_t
-
\nabla_y c(X_t,Y_t)
- 
\epsilon\nabla_y V(Y_t)
\right|^p
\right]
\,dt
<\infty.
\]

\item \label{weak-sol:MV-a:sde}
The following system of SDEs holds $\Prob$-almost surely:
\begin{equation*}
\left\{
\begin{aligned}
dX_t
&=
\bigl(
\hat A_t
-\nabla_x c(X_t,Y_t)
-\epsilon \nabla U(X_t)
\bigr)\,dt
+\sqrt{2\epsilon}\,dW_t,
\\
dY_t
&=
\bigl(
\hat B_t
-\nabla_y c(X_t,Y_t)
-\epsilon \nabla V(Y_t)
\bigr)\,dt
+\sqrt{2\epsilon}\,dB_t.
\end{aligned}
\right.
\end{equation*}

\item
For almost every $t\geq 0$,
\[
\hat A_t
=
\E^{\Prob}[\nabla_xc(X_t,Y_t)\vert X_t],
\qquad
\hat B_t
=
\E^{\Prob}[\nabla_yc(X_t,Y_t)\vert Y_t],
\qquad \Prob\text{-a.s.}
\]
\end{enumerate}
\end{definition}

Uniqueness is understood in law: for a given initial distribution, any two weak solutions have the same law.

Once weak existence is established, an important difference from \cite{CLP} arises: there, the flow of measures satisfies $P_t\in\Pi(\mu,\nu)$ for all $t\geq 0$, which is no longer true in our setting.

\subsection{Weak existence and uniqueness}

The proof follows that of \cite[Proposition~3.8]{CLP}; we give a condensed version for completeness. For $T>0$, a path law $\PathLaw{P}$ on $\Omega_T$ with coordinate process $(X_t,Y_t)$, and $0\le t\le T$, we set
\begin{equation}\label{eq:path-law-conventions}
\Omega_T:=C([0,T],\Z),\qquad P_t:=\PathLaw{P}\circ(X_t,Y_t)^{-1},\qquad \PathLaw{P}^t:=\PathLaw{P}\circ(\omega\mapsto\omega\vert_{[0,t]})^{-1},
\end{equation}
so that $P_t$ is the time marginal of $\PathLaw{P}$ and $\PathLaw{P}^t$ its restriction to $[0,t]$.

\begin{proposition}
\label{prop:well-posedness}
Under Assumption~\ref{ass:regularity}, for every
$P_0\in\Pc(\Z)$, equation \eqref{eq:dynamics} admits a unique weak solution.
\end{proposition}

\begin{proof}
Fix $T>0$. On $\Omega_T$, define
\begin{equation}\label{eq:cost-gradient-bound}
 c_x:=\nabla_xc,\quad c_y:=\nabla_yc,\quad K_c:=\max(\|c_x\|_\infty,\|c_y\|_\infty)<\infty,
\end{equation}
where boundedness follows from Assumption~\ref{ass:regularity} and compactness. Let $\PathLaw{Q}^0$ be Wiener measure with initial law $P_0$ and diffusion coefficient $\epsilon$.

For $\PathLaw{Q}\in\Pc(\Omega_T)$, \cite[Proposition~5.1]{BrunickShreve2013} (see also \cite[Proposition~3.3]{CLP}) provides jointly Borel versions of $\E^{Q_t}[c_x\vert X=x]$ and $\E^{Q_t}[c_y\vert Y=y]$. Let $\Phi(\PathLaw{Q})$ be the path law of the frozen equation started from $P_0$,
\begin{equation}\label{eq:frozen-equation}
d(X_t,Y_t)=\Drift[Q_t](X_t,Y_t)\,dt+\sqrt{2\epsilon}\,d(W_t,B_t),
\end{equation}
with $\Drift$ as in~\eqref{eq:drift-definition}. By~\eqref{eq:cost-gradient-bound} and the bounds on $\nabla U,\nabla V$, this drift is bounded uniformly in $\PathLaw{Q}$, so Girsanov's theorem yields weak well-posedness and, for every $q<\infty$,
\begin{equation}\label{eq:girsanov-uniform-bound}
K_q^{\mathrm{Gir}}:=\sup_{\PathLaw{Q}\in\Pc(\Omega_T)}\big\|d\Phi(\PathLaw{Q})/d\PathLaw{Q}^0\big\|_{L^q(\PathLaw{Q}^0)}<\infty.
\end{equation}

Let $\mathcal K^{\mathrm{fp}}$ be the closed convex hull of $\Phi(\mathcal K^{\mathrm{fp}}_0)$, where
\begin{equation}\label{eq:fixed-point-class}
\mathcal K^{\mathrm{fp}}_0:=\big\{\PathLaw{Q}\ll\PathLaw{Q}^0:\ \|d\PathLaw{Q}/d\PathLaw{Q}^0\|_{L^2(\PathLaw{Q}^0)}\le K_2^{\mathrm{Gir}}\big\}.
\end{equation}
By weak compactness of $\mathcal K^{\mathrm{fp}}_0$ in $L^2$ and~\eqref{eq:girsanov-uniform-bound}, $\mathcal K^{\mathrm{fp}}\subset\mathcal K^{\mathrm{fp}}_0$ is weakly compact and convex, and $\Phi(\mathcal K^{\mathrm{fp}})\subset\mathcal K^{\mathrm{fp}}$.

All equations have the same nondegenerate diffusion and a uniformly bounded drift. Standard H\"older estimates, as in Theorem~6.2.2 of~\cite{BogachevKrylovRocknerShaposhnikov2015}, give uniform upper and H\"older bounds on $[\delta,T]\times\Z$. Writing $p_t^{\Phi(\PathLaw{Q})}$ for the time-marginal density, for some $\theta\in(0,1)$,
$$
 |p_t^{\Phi(\PathLaw{Q})}(z)-p_{t'}^{\Phi(\PathLaw{Q})}(z')|
 \le C_\delta\bigl(|t-t'|^{\theta/2}+d_\Z(z,z')^\theta\bigr),\qquad \delta\le t,t'\le T,\quad z,z'\in\Z.
$$
These bounds, and the corresponding positive lower bounds, are preserved by convex combinations and weak limits.

Thus, if $\PathLaw{Q}^n\to\PathLaw{Q}$ weakly in $\mathcal K^{\mathrm{fp}}$ from~\eqref{eq:fixed-point-class}, their time-marginal densities converge uniformly on $[\delta,T]\times\Z$; in particular $Q_t^n\to Q_t$ in total variation for each $t>0$. Consequently, as in Proposition~3.4 of~\cite{CLP},
\begin{equation*}
 \E^{Q_t^n}[c_x\vert X]\longrightarrow\E^{Q_t}[c_x\vert X]\quad\text{in }L^2(Q_t),\qquad\text{for a.e. }t>0.
\end{equation*}
The same holds for $c_y$ conditioned on $Y$. The density bounds also give convergence in $L^2((\Phi(\PathLaw{Q}))_t)$ on $[\delta,T]$. By Girsanov's theorem, with entropy~\eqref{eq:relative-entropy},
$$
 \Ent(\Phi(\PathLaw{Q})\mid\Phi(\PathLaw{Q}^n))
 \le\frac1{4\epsilon}\int_0^T
 \E^{\Phi(\PathLaw{Q})}\!\left[
 |\Drift[Q_t](X_t,Y_t)-\Drift[Q_t^n](X_t,Y_t)|^2\right]dt.
$$
The right hand side converges toward $0$ by dominated convergence.
Thus the map $\Phi$ defined by~\eqref{eq:frozen-equation} is weakly continuous on $\mathcal K^{\mathrm{fp}}$, and the fixed-point theorem yields weak existence.
\end{proof}

For the uniqueness, we need the following lemma.

\begin{lemma}\label{lem:autonomous-marginals}
Let $\PathLaw{P}$ be the path law of a weak solution of~\eqref{eq:dynamics}. For $t>0$, $P_t$ and its marginals $\mu_t,\nu_t$ have densities, and $(\mu_t,\nu_t)$ is the unique solution, starting from the marginals $(\mu_0,\nu_0)$ of $P_0$, of
\begin{equation}\label{eq:autonomous-marginals-existence}
\partial_t\mu_t=\epsilon\nabla_x\!\cdot(\nabla_x\mu_t+\mu_t\nabla U),\qquad
\partial_t\nu_t=\epsilon\nabla_y\!\cdot(\nabla_y\nu_t+\nu_t\nabla V).
\end{equation}
\end{lemma}

\begin{proof}
    It follows readily from It\^o's formula and properties of conditional expectation that $\mu_t$ and $\nu_t$ are weak solutions of the Fokker--Planck equations \eqref{eq:autonomous-marginals-existence}. Furthermore, by the hypotheses on $U$ and $V$, these equations admit a unique weak solution each, by Theorem $9.4.3$ of \cite{BogachevKrylovRocknerShaposhnikov2015}.
\end{proof}

We now turn to the proof of uniqueness in Proposition~\ref{prop:well-posedness}.

\begin{proof}
The uniqueness follows from the same line of argument as in~\cite{CLP}.
Let $\PathLaw{P}_1$ and $\PathLaw{P}_2$ be two weak solution laws, with the conventions~\eqref{eq:path-law-conventions} and time marginals $P_{1,t},P_{2,t}$. Recall $c_x,c_y$ from~\eqref{eq:cost-gradient-bound}. The drift $\Drift[P_{i,t}]$ is uniformly bounded, so Girsanov's theorem applies and gives
\begin{align*}
 \Ent(\PathLaw{P}_1^t\mid\PathLaw{P}_2^t)
 &=\frac1{4\epsilon}\int_0^t\E^{\PathLaw{P}_1}
 \bigl[|\Drift[P_{1,s}](X_s,Y_s)-\Drift[P_{2,s}](X_s,Y_s)|^2\bigr]\,ds\\
 &=\frac1{4\epsilon}\int_0^t\E^{P_{1,s}}
 \Bigl[|\E^{P_{1,s}}[c_x\vert X]-\E^{P_{2,s}}[c_x\vert X]|^2
 +|\E^{P_{1,s}}[c_y\vert Y]-\E^{P_{2,s}}[c_y\vert Y]|^2\Bigr]ds.
\end{align*}
Write $P_{i,s;\Y\mid x}$ and $P_{i,s;\X\mid y}$ for the corresponding conditional laws, as in~\eqref{eq:conditional-laws}, and set $\TV(P,Q):=\sup_A|P(A)-Q(A)|$. Since $c_x$ is bounded,
\[
\E^{P_{1,s}}\bigl[|\E^{P_{1,s}}[c_x\vert X]-\E^{P_{2,s}}[c_x\vert X]|^2\bigr]
\le4\|c_x\|_\infty^2\,\E^{P_{1,s}}\bigl[\TV^2(P_{1,s;\Y\mid X},P_{2,s;\Y\mid X})\bigr],
\]
and the same bound holds for $c_y$ with the roles of $X$ and $Y$ exchanged.
By Pinsker's inequality, $\TV^2(P,Q)\le \Ent(P\mid Q)/2$. By Lemma~\ref{lem:autonomous-marginals}, both solutions have the same marginal flows $\mu_s,\nu_s$, so the chain rule for entropy gives $\E^{\mu_s}[\Ent(P_{1,s;\Y\mid x}\mid P_{2,s;\Y\mid x})]=\Ent(P_{1,s}\mid P_{2,s})=\E^{\nu_s}[\Ent(P_{1,s;\X\mid y}\mid P_{2,s;\X\mid y})]$. Hence, using the bound $K_c$ from~\eqref{eq:cost-gradient-bound},
\begin{align*}
\Ent(\PathLaw{P}_1^t\mid\PathLaw{P}_2^t)
&\le\frac{K_c^2}{2\epsilon}\int_0^t\Bigl(\E^{\mu_s}[\Ent(P_{1,s;\Y\mid x}\mid P_{2,s;\Y\mid x})]+\E^{\nu_s}[\Ent(P_{1,s;\X\mid y}\mid P_{2,s;\X\mid y})]\Bigr)ds\\
&=\frac{K_c^2}{\epsilon}\int_0^t\Ent(P_{1,s}\mid P_{2,s})\,ds
\le\frac{K_c^2}{\epsilon}\int_0^t\Ent(\PathLaw{P}_1^s\mid\PathLaw{P}_2^s)\,ds,
\end{align*}
the last step by data processing. Since the initial laws coincide, Gronwall's inequality yields $\PathLaw{P}_1=\PathLaw{P}_2$.
\end{proof}

We conclude this subsection with the following lemma, which gives some estimates on the density of the weak solution of \eqref{eq:dynamics} and is a straightforward consequence of standard bootstrap arguments.

\begin{lemma}
\label{lem:asymptotic-compactness}
Under Assumption~\ref{ass:regularity}, every global weak solution
from an arbitrary $P_0\in\Pc(\Z)$ has a strictly positive density $p_t$ at every
positive time.  For every $\tau>0$ and $0<\theta<\alpha_0$, there are
constants $c_\tau,C_{\tau,\theta}>0$, depending only on the fixed model data,
$\tau$, and, for $C_{\tau,\theta}$, $\theta$, such that, uniformly over
initial laws and solutions,
\begin{equation}
 c_\tau\le p_t\le C_{\tau,\theta},
 \qquad
 \sup_{t\ge\tau}\|p_t\|_{C^{1,\theta}(\Z)}\le C_{\tau,\theta}.
 \label{eq:compact-bounds}
 \vspace{-1.5mm}
\end{equation}
More precisely, the parabolic H\"older bounds satisfy
\[
 \sup_{t\ge\tau}
 \|p\|_{C^{(1+\theta)/2,\,1+\theta}([t,t+1]\times\Z)}
 \le C_{\tau,\theta}.
  \vspace{-1.5mm}
\]
\end{lemma}

\begin{proof}
By Proposition~\ref{prop:well-posedness}, equation \eqref{eq:dynamics} admits a unique weak solution, which also admits a density.
First, the lower and upper bounds are consequences of the usual estimates for Kolmogorov--Fokker--Planck equations; see for instance \cite{Aronson} or \cite{BogachevKrylovRocknerShaposhnikov2015}.

Choose $\theta<\beta<1$, since the diffusion matrix
is $\epsilon I$, the drift is bounded, and the nonnegative density has
finite spatial mass, the estimates
\cite[Corollary~6.4.5 and Theorem~6.2.2(ii)]{BogachevKrylovRocknerShaposhnikov2015}
give a locally uniform $C^{\beta/2,\beta}$ bound for $p$, in the sense that 
$$
\|p\|_{C^{\beta/2,\beta}([\tau,T]\times\mathbb T^d)}
\le C_{\tau,T,\beta,\epsilon},
\qquad 0<\tau<T.
$$
Consequently, the conditional expectations, seen as Borel functions : 
\[
 A_p(t,x):=\E^{P_t}[\nabla_xc\mid X=x]
 =\frac{\int_\Y\nabla_xc(x,y)p_t(x,y)\,dy}
        {\int_\Y p_t(x,y)\,dy}
\]
also has a locally uniform $C^{\theta/2,\theta}$ bound.
The same argument applies to the conditional expectation in $y$.
Since $c,U,V\in C^{1,\alpha_0}$, the drift
$(s,z)\mapsto\Drift[P_s](z)$, and hence the flux
$p_s\Drift[P_s]$, therefore have uniform
$C^{\theta/2,\theta}$ bounds.

Applying the divergence-form Schauder estimates \cite[Theorem~2.1 and Section~4.1, with $m=1$]{DZ} to $\partial_sp-\epsilon\Delta p=-\operatorname{div}(p\Drift[P_s])$ on these intervals gives $\sup_{t\ge\tau}\|p\|_{C^{(1+\theta)/2,1+\theta}([t,t+1]\times\Z)}\le C_{\tau,\theta,\epsilon}$, which, together with the lower and upper bounds, proves~\eqref{eq:compact-bounds}.
\end{proof}

\subsection{Weak chain rule for entropy}

We record the following chain-rule identity for the relative entropy~\eqref{eq:relative-entropy} between two
solutions of linear Fokker--Planck equations with possibly different, merely
H\"older-continuous drifts. A version of this identity, in a different setting,
is discussed in Section $4.4$ of \cite{CLP}.

\begin{lemma}\label{lem:weak-chain-rule}
Let $\theta,\theta'\in(0,1)$ and let $b^1,b^2:[s,t]\times\Z\to\mathbb R^{d_x+d_y}$
be bounded, $\theta$-H\"older continuous in the space variable uniformly in
$u\in[s,t]$, and locally $\theta'$-H\"older continuous in time, uniformly in
$z\in\Z$ -- exactly the regularity established for the field $-\Drift[P_t]$ in the
proof of Lemma~\ref{lem:asymptotic-compactness}. Let $Q^1$ and $Q^2$, with positive densities $q^1,q^2$, solve
\[
 \partial_tq^i=\epsilon\Delta q^i+\operatorname{div}(q^ib^i),\qquad i=1,2,
\]
on $[s,t]\times\Z$ for some $0<s\le t$, enjoying the regularity conclusions
of Lemma~\ref{lem:asymptotic-compactness} (positivity, uniform upper and lower
bounds, and $C^{1,\theta}(\Z)$ regularity in space, uniformly on $[s,t]$). Then
\begin{equation}\label{eq:weak-chain-rule}
\Ent(Q^1_t\mid Q^2_t)=\Ent(Q^1_s\mid Q^2_s)-\int_s^t\!\!\int_\Z\bigl(b^1-b^2+\epsilon\nabla_z\ell_u\bigr)\cdot\nabla_z\ell_u\,q^1_u\,dz\,du,
\qquad \ell_u:=\log\tfrac{q^1_u}{q^2_u}.
\end{equation}
\end{lemma}

\begin{proof}
As discussed in the proof of Lemma $4.1$ of \cite{CLP}, we first regularize the density with respect to space.
Let $q^{i,\delta}=q^i*\rho_\delta$ and thanks to the bounds provided by Lemma \ref{lem:asymptotic-compactness}
\[
 r^{i,\delta}=(q^ib^i)*\rho_\delta-q^{i,\delta}b^i,
 \qquad \|r^{i,\delta}\|_\infty\le C\delta^\theta.
\]
For fixed $\delta$, the equation makes $q^{i,\delta}$ absolutely
continuous in time with smooth spatial values. Differentiating its
entropy, and integrating by parts only in space, gives
\[
 \begin{aligned}
 \frac d{du}\Ent(Q^{1,\delta}_u\mid Q^{2,\delta}_u)
 ={}&-\int_\Z
 (b^1-b^2+\epsilon\nabla R^\delta)\cdot\nabla R^\delta
                                      \,q^{1,\delta}\,dz+E_\delta(u),\\
 E_\delta(u)={}&-\int_\Z r^{1,\delta}\cdot\nabla R^\delta\,dz
 +\int_\Z\frac{q^{1,\delta}}{q^{2,\delta}}
                      r^{2,\delta}\cdot\nabla R^\delta\,dz,
 \qquad R^\delta=\log(q^{1,\delta}/q^{2,\delta}).
 \end{aligned}
\]
The lower bounds and uniform spatial $C^{1,\theta}$ bounds imply
$q^{i,\delta}\to q^i$ and $\nabla q^{i,\delta}\to\nabla q^i$
uniformly on $[s,t]\times\Z$, and $\sup_u|E_\delta(u)|\to0$.
Integrating the identity on any subinterval and passing to the limit
proves~\eqref{eq:weak-chain-rule} and absolute continuity.

\end{proof}

\subsection{Wasserstein-to-entropy regularization}
\label{sec:wasserstein-initialization}

We now establish an estimate which is a consequence of the smoothing effect of the Markov semigroup.

\begin{proposition}[Wasserstein-to-entropy regularization]
\label{prop:wasserstein-smoothing}
Let $\pi$ be the entropic optimal coupling of \eqref{eq:pi-form}, and for
$1\le a\le\infty$ set
\[
 \lambda_a:=\frac12+\frac{d(a-1)}{2a},
 \qquad\text{with }\lambda_\infty:=\frac{1+d}{2}.
\]
For every $1\le a\le\infty$ there exists a constant
$C_a=C_a(\varepsilon,U,V,c)>0$ such that
\[
 \|p_t-\pi\|_{\mathbb L^a(\Z)}
 \le C_a\,t^{-\lambda_a}\,\Wass{1}(P_0,\pi),
 \qquad 0<t\le1 .
\]
Consequently, for every $0 < \tau < 1$ there exists a constant
$C_\tau=C(\tau,\varepsilon,U,V,c)$ such that
\begin{equation}
 \Ent(P_\tau\mid\pi)
 \le C_\tau\,\Wass{1}(P_0,\pi)^2
 \le C_\tau\,\Wass{2}(P_0,\pi)^2,
 \label{eq:ws-entropy-smoothing}
\end{equation}
and one may take $C_\tau=C\,\tau^{-(1+d/2)}$ for $0<\tau\le1$.
\end{proposition}

\begin{proof}
Throughout, $\|\cdot\|_a$ denotes the norm of $\mathbb L^a(\Z)$, and $C$
denotes a constant depending only on $(\varepsilon,U,V,c)$, and possibly on
$a$, whose value may change from line to line. Set
\[
 w_t:=p_t-\pi,\qquad e_a(t):=\|w_t\|_a,\qquad
 W:=\Wass{1}(P_0,\pi),\qquad b_*:=\Drift[\pi].
\]
Since $\pi$ is a stationary solution of \eqref{eq:dynamics}, $w$ solves
\[
 \partial_t w_t=\varepsilon\Delta w_t-\operatorname{div}\bigl(b_*w_t+F[p_t]\bigr),
 \qquad F[p_t]:=\bigl(\Drift[p_t]-b_*\bigr)p_t .
\]

\smallskip
We first claim that
\begin{equation}\label{eq:flux-bound}
 \|F[p_t]\|_a\le C\,\|p_t-\pi\|_a,
 \qquad 1\le a\le\infty,
\end{equation}
with $C$ independent of $a$ and $t$. Let
$A_t:=\E^{p_t}[\nabla_xc\mid X]$, $A_\pi:=\E^{\pi}[\nabla_xc\mid X]$,
$m_t:=(p_t)_\X$ and $\mu:=\pi_\X$, where measures are identified with their
densities. 
The $x$-component of $F[p_t]$ is $(A_t-A_\pi)p_t$, and a direct computation gives
\[
m_t(A_t-A_\pi)=\int_\Y(\nabla_xc-A_\pi)(p_t-\pi)\,dy,\qquad
\mu(A_t-A_\pi)=\int_\Y(\nabla_xc-A_t)(p_t-\pi)\,dy.
\]
Since $\nabla_xc$ is bounded, the first identity yields $\|(A_t-A_\pi)p_t\|_1=\|m_t(A_t-A_\pi)\|_{\mathbb L^1(\X)}\le 2\|\nabla_xc\|_\infty\|p_t-\pi\|_1$. For general $a$, the second gives
\[
(A_t-A_\pi)p_t=(A_t-A_\pi)(p_t-\pi)+\tfrac{\pi}{\mu}\textstyle\int_\Y(\nabla_xc-A_t)(p_t-\pi)\,dy.
\]
The first term is bounded pointwise by $2\|\nabla_xc\|_\infty|p_t-\pi|$.
For the second, $\pi/\mu$ is bounded since $\pi$ is bounded and $\mu$ is
bounded below by a positive constant, and H\"older's inequality
gives
\[
 \Bigl|\int_\Y(\nabla_xc-A_t)(p_t-\pi)\,dy\Bigr|
 \le 2 C_a \|\nabla_xc\|_\infty\
 \Bigl(\int_\Y|p_t-\pi|^a\,dy\Bigr)^{1/a}.
\]
This yields the bound in $\mathbb L^a$
uniformly in $a$. The $y$-component is treated in the same way, which
proves~\eqref{eq:flux-bound}.

\smallskip
\noindent\emph{Semigroup estimates.} Let $(S_h)_{h\ge0}$ be the semigroup
generated by $\varepsilon\Delta-\operatorname{div}(b_*\,\cdot\,)$. The Gaussian
kernel bounds and the gradient estimates in the initial variable of
\cite[Theorem~1.2(i)--(ii)]{MenozziPesceZhang2021}, combined with Young's
inequality, give, for $1\le a\le r\le\infty$ and $0<h\le1$,
\begin{equation}\label{eq:semigroup-lp}
 \|S_hf\|_r\le C_{a,r}\,h^{-\eta}\|f\|_a,\qquad
 \|S_h\operatorname{div}A\|_r\le C_{a,r}\,h^{-1/2-\eta}\|A\|_a,
 \qquad
 \eta:=\frac d2\Bigl(\frac1a-\frac1r\Bigr).
 \vspace{-2mm}
\end{equation}

\smallskip
\noindent\emph{Strategy.} For $1\le a\le\infty$, consider the statement
\begin{equation}\label{eq:Ea}
\text{there exists } C_a>0 \text{ such that } e_a(t)\le C_a\,t^{-\lambda_a}\,W \text{ for all } 0<t\le1 \text{ and all } P_0, \tag{$E_a$}
\end{equation}
and set $I:=\{a\in[1,\infty]:(E_a)\text{ holds}\}$. The first claim of the proposition is that $I=[1,\infty]$.
\smallskip
\noindent\emph{Step 0: $I$ is an interval.} Let $a_0<a<a_1$ with
$a_0,a_1\in I$, and let $\theta\in(0,1)$ be such that
$\frac1a=\frac{\theta}{a_0}+\frac{1-\theta}{a_1}$. By H\"older's inequality,
\[
 e_a(t)\le e_{a_0}(t)^{\theta}\,e_{a_1}(t)^{1-\theta}
 \le C\,t^{-(\theta\lambda_{a_0}+(1-\theta)\lambda_{a_1})}\,W
 = C\,t^{-\lambda_a}\,W,
\]
where the last equality holds because $\lambda_a=\frac{1+d}{2}-\frac{d}{2a}$
is affine in $1/a$. Hence $a\in I$.

\smallskip
\noindent\emph{Step 1: $1\in I$.}
For $z \in \Z$,
$$
 S_t (P_0 - \pi) (z )
 =
 \int S_t ( \delta_{w}) ( z ) P_0(\mathrm{d}w)
-
 \int S_t ( \delta_{w}) ( z ) \pi (\mathrm{d}w)
 =
  \int S_t ( \delta_{w}) ( z ) (P_0 - \pi )(\mathrm{d}w)
 .
$$
By the regularization property of the Markov semigroup (see the introduction of \cite{MenozziPesceZhang2021} or \cite[Theorem~1.1]{ChenHuXieZhang2017}), $w\mapsto S_t(\delta_w)(z)$ is $Ct^{-1/2}$-Lipschitz, hence $|S_t(P_0-\pi)(z)|\le Ct^{-1/2}W$, where $W:=\Wass{1}(P_0,\pi)$ as in~\eqref{eq:wasserstein-distances}. Combining Duhamel's formula $w_t=S_t(P_0-\pi)-\int_0^t S_{t-s}\operatorname{div}F[p_s]\,ds$ with~\eqref{eq:semigroup-lp} for $a=r=1$ and~\eqref{eq:flux-bound}, we get
\[
e_1(t)\le Ct^{-1/2}W+C\int_0^t(t-s)^{-1/2}e_1(s)\,ds,\qquad 0<t\le1.
\]
The Henry--Gronwall inequality \cite[Theorem~3]{Medved1997}, which applies since $s\mapsto s^{-1/2}$ is integrable on $(0,1)$, then yields $e_1(t)\le Ct^{-1/2}W$ for $0<t\le1$.

\smallskip
\noindent\emph{Step 2: there exists $\delta>0$ such that $a\in I$ implies
$a+\delta\in I$ for every finite $a\ge1$.} Set
\[
 \delta:=\frac1{2d-1},\qquad r:=a+\delta,\qquad
 \eta_a:=\frac d2\Bigl(\frac1a-\frac1r\Bigr)=\frac{d\delta}{2a(a+\delta)} .
\]
Since $a\ge1$, we have
$\eta_a\le\frac{d\delta}{2(1+\delta)}=\frac14$, and moreover
$\lambda_a+\eta_a=\lambda_r$. Assume $\eqref{eq:Ea}$. Writing Duhamel's formula
between $t/2$ and $t$ and applying \eqref{eq:semigroup-lp} and
\eqref{eq:flux-bound}, we obtain, for $0<t\le1$,
\[
 \begin{aligned}
 e_r(t)&\le C\,t^{-\eta_a}\,e_a(t/2)
       +C\int_{t/2}^t(t-s)^{-1/2-\eta_a}\,e_a(s)\,ds\\
 &\le C\,W\bigl(t^{-\lambda_a-\eta_a}+t^{1/2-\eta_a-\lambda_a}\bigr)
 \le C\,t^{-\lambda_r}\,W,
 \end{aligned}
\]
where we used $e_a(s)\le C_a(t/2)^{-\lambda_a}W$ for $s\in[t/2,t]$, the
integrability of $(t-s)^{-1/2-\eta_a}$ (as $\eta_a\le\frac14$), and $t\le1$.
Hence $(E_r)$ holds, with a constant depending only 
on $C_a$.

\smallskip
\noindent\emph{Conclusion: $I=[1,\infty]$.} 
In particular, there exists $C_2$ such that
\begin{equation}\label{eq:l2-regularization}
 e_2(t)\le C\,t^{-1/2-d/4}\,W,\qquad 0<t\le1.
\end{equation}
Let $\pi_*:=\min_\Z\pi>0$. Since
$\log u\le u-1$ and $\int_\Z p_t\,dz=\int_\Z\pi\,dz=1$,
\[
 \Ent(P_t\mid\pi)
 \le\int_\Z p_t\Bigl(\frac{p_t}{\pi}-1\Bigr)dz
 =\int_\Z\frac{(p_t-\pi)^2}{\pi}\,dz
 \le\pi_*^{-1}\,e_2(t)^2 .
\]
Together with \eqref{eq:l2-regularization}, this proves
\eqref{eq:ws-entropy-smoothing} for $0<\tau\le1$, with
$C_\tau=C\tau^{-(1+d/2)}$. Finally, since $\Wass{1}\le\Wass{2}$, we get the desired inequality.
\end{proof}

% !TeX root = ../Draft.tex
\section{Local estimates}
\label{sec:local-static-estimates}
This section provides an upper bound of $\Ent(P_t \mid \pi)$, whenever the marginal density ratios are sufficiently close to $1$ in $C^{1,\sigma}$, for $0<\sigma<\alpha_0$, in terms of the projected Fisher information $\PFI(P_t \mid \pi)$ from~\eqref{eq:projected-fisher} and the relative entropies $\Ent(\mu_t \mid \mu)+\Ent(\nu_t \mid \nu)$. The central idea is the decomposition
\begin{align*}
    \Ent(P_t \mid \pi)=\Ent(P_t \mid \Bridge{\mu_t}{\nu_t})+\Ent( \Bridge{\mu_t}{\nu_t} \mid \pi), \qquad \PFI(P_t \mid \pi)= \PFI(P_t \mid \Bridge{\mu_t}{\nu_t})
\end{align*}
where $\Bridge{\mu_t}{\nu_t}$, constructed in~\eqref{eq:bridge-scaling}, is the minimizer of
\(
    \min_{Q \in \Pi(\mu_t,\nu_t)}\Ent(Q \mid \pi). 
\)
By \cite[Theorem~2.1]{Nutz2022}, the minimizer $\Bridge{\mu_t}{\nu_t}$ has the Gibbs form $d\Bridge{\mu_t}{\nu_t}=e^{a(x)+b(y)}\,d\pi$; we call such couplings \emph{martingale bridges}. Lemma~\ref{lem:bridge-smooth} shows that when $\|d\mu_t/d\mu-1\|_{C^{1,\sigma}}+\|d\nu_t/d\nu-1\|_{C^{1,\sigma}}$ is small, the potentials $a,b$ are regular enough to yield the projected log-Sobolev inequality of Corollary~\ref{cor:nearby-plsi},
\[
\Ent(P_t\mid\Bridge{\mu_t}{\nu_t})\le\tfrac{1}{2\rho}\,\PFI(P_t\mid\Bridge{\mu_t}{\nu_t}),
\]
which we apply in Section~\ref{sec:convergence-rate}, once $(\mu_t,\nu_t)$ are close enough to $(\mu,\nu)$.
\subsection{Existence and properties of bridges}

Fix $0<\sigma<\alpha_0$ and let $\alpha=h\mu$ and $\beta=k\nu$ be
probability measures with $(h,k)$ close to $(1,1)$ in $C^{1,\sigma}$ or $L^\infty$ norm.
We use $(a\oplus b)(x,y):=a(x)+b(y)$ and define
\begin{equation}
 r_\pi:=\frac{d\pi}{d(\mu\otimes\nu)},\qquad r_*:=\min_\Z r_\pi>0.
 \label{eq:bridge-reference-density}
\end{equation}
We first construct a coupling by multiplicative scaling,
\begin{equation}
 d\Bridge{\alpha}{\beta}=e^{a(x)+b(y)}d\pi,
 \label{eq:bridge-scaling}
\end{equation}
and then identify it as the unique minimizer of $\Ent(\cdot\mid\pi)$ over
$\Pi(\alpha,\beta)$.  The construction in the $C^{1,\sigma}$ norm provides sufficient regularity 
used for proving the local projected logarithmic Sobolev inequality, and its
$L^\infty$ version suffices for estimating $\Ent(\Bridge{\alpha}{\beta} \mid \pi)$.  

\begin{lemma}\label{lem:bridge-smooth}
Let $\mathcal{S}$ be either $C^{1,\sigma}(\X) \times C^{1,\sigma}(\Y)$ or $L^\infty(\X) \times L^\infty(\Y)$ and take closed subsets
\begin{align*}
    \mathcal E_{\mathrm{marg}}:= \left\{(h,k) \in \mathcal{S}: \, \int h \, d\mu=\int k \, d\nu  \right\}, \quad \mathcal{R}:=\left\{(a,b) \in \mathcal{S}: \, \int a \, d\mu=0 \right\}
\end{align*}
There exists a neighborhood $\mathcal{U} $ of $(1,1)$ in $\mathcal E_{\mathrm{marg}}$ and a smooth functional $$\mathcal{G}: \mathcal{U} \to \mathcal{R};\, \, (h,k) \mapsto (a,b) $$  such that          $e^{a\oplus b} \pi \in \Pi(h\mu, k \nu)$. In particular, the coupling $\Bridge{h\mu}{k\nu}$ in \eqref{eq:bridge-scaling} is well-defined for $(h,k) \in \mathcal{U}$. 
\end{lemma}

\begin{proof}
We prove the result by Banach implicit function theorem. Let us first construct the inverse of $\mathcal{G}$. Recall $r_\pi,r_*$ from~\eqref{eq:bridge-reference-density} and define, for any $(u,v)\in\mathcal S$,
\begin{align}\label{eq:conditional-operators}
    Kv(x):=& \E^{\pi}[v(Y) \,\vert \,X=x]=\int_{\Y} r_{\pi}(x,y)v(y) \, \nu(dy), \\
    K^*u(y):=& \E^{\pi}[u(X) \,\vert \,Y=y]=\int_{\X} r_{\pi}(x,y)u(x) \, \mu(dx). \notag
\end{align}
For any $(a,b) \in \mathcal{R}$, let $\mathcal{F}(a,b)$ be the marginal densities of $e^{a \oplus b}\pi$, i.e. 
\begin{equation*}
 \mathcal F(a,b)=\left(\frac{ (e^{a \oplus b} \pi)_{\X}}{d\mu}, \frac{(e^{a \oplus b} \pi)_{\Y}}{d \nu} \right) =
 \bigl(e^aK(e^b),\ e^bK^*(e^a)\bigr).
\end{equation*}
The functional $\mathcal{F}:\mathcal{R} \to \mathcal E_{\mathrm{marg}}$ is smooth in the sense of Fr\'{e}chet differentiability. At $(0,0)$, for any direction $(u,v) \in \mathcal{R}$,
\begin{equation}
    D\mathcal{F}_{(0,0)}(u,v)=\lim\limits_{t \to 0} \frac{\mathcal{F}(tu,tv)-\mathcal{F}(0,0)}{t}
 =(u+Kv,\ v+K^*u).
 \label{eq:linearized-scaling}
\end{equation}

It remains to prove that $D\mathcal{F}_{(0,0)}: \mathcal{R} \to \mathcal E_{\mathrm{marg}}$ is a bounded linear isomorphism with bounded inverse. Then according to \cite[Corollary 15.1]{Deimling1985}, there exists an inverse $\mathcal{G}$ of $\mathcal{F}$ that satisfies the desired properties. 

As $r_{\pi} \in C^{1,\alpha_0}(\Z)$, $D\mathcal{F}_{(0,0)}$ is a bounded linear operator. If $r_*=1$, then $r_{\pi} \equiv 1$. Note that 
\begin{align*}
u(x)+Kv(x)=&u(x)+\int_{\Y} v(y) \, \nu(dy), \\
v(y)+ K^*u(y)=& v(y) + \int_{\X} u(x) \, \mu(dx)=v(y).
\end{align*}
Therefore, according to \eqref{eq:linearized-scaling}, $D\mathcal{F}_{(0,0)}$ is bijective with bounded inverse. 

If $r_* \in (0,1)$, write $\tilde r= \frac{r_{\pi}-r_*}{1-r_*}$ and $\tilde \pi= \tilde r \, (\mu \otimes \nu)$. By Jensen's inequality and equation \eqref{eq:conditional-operators}, $\lVert Kv \rVert_{L^2(\mu)}\leq \lVert v \rVert_{L^2(\nu)}$. Moreover, since $\int u \, d \mu=0$, 
\begin{align*}
    K^*u(y)= (1-r_*) \int_{\X} \tilde r (x,y) u(x) \, \mu(dx)= (1-r_*) \E^{\tilde \pi}[u(X) \,\vert \,Y=y],
\end{align*}
and hence $\lVert K^* u \rVert_{L^2(\nu)} \leq (1-r_*) \lVert u \rVert_{L^2(\mu)}$. 

% Now, suppose that there exist $(u,v) \in \mathcal{R}$ such that $D \mathcal{F}_{(0,0)}(u,v)=(0,0)$. Then from \eqref{eq:linearized-scaling}, we get $u=KK^*u$, which is impossible as $KK^*:L^2(\mu) \to L^2(\nu)$ is a strict contraction. Hence, $D \mathcal{F}_{(0,0)}$ is injective. 

Take an arbitrary $(\xi,\eta) \in \mathcal E_{\mathrm{marg}}$. We claim that there exists a unique pair $(u,v)\in\mathcal{R}$ such that $D\mathcal{F}_{(0,0)}(u,v)=(\xi,\eta)$. According to \eqref{eq:linearized-scaling}, it is equivalent to finding $u \in \mathcal{B}:=\{ u \in L^2(\mu): \,\int_{\X} u \, d\mu=0 \}$ such that 
\begin{align}\label{eq:scaling-linear-equation}
    u-KK^*u=\xi-K\eta.
\end{align}
It is clear that $T(u):= KK^*u+\xi-K\eta$ is a contraction on the Banach space $\mathcal{B}$. Therefore, there exists a unique solution to \eqref{eq:scaling-linear-equation} in $\mathcal{B}$ thanks to the Banach fixed point theorem. Since $r_\pi \in C^{1,\alpha_0}(\Z)$, $KK^*u \in C^{1,\sigma}(\X)$ for any $u \in \mathcal{B}$. It follows that the solution $u \in C^{1,\sigma}(\X)$ as well, and hence
 $D\mathcal{F}_{(0,0)}$ is a bijection. 

Let us show that the inverse of $D\mathcal{F}_{(0,0)}$ is bounded. Suppose $u$ satisfies \eqref{eq:scaling-linear-equation}. Then it holds that 
\begin{align*}
    r_* \lVert u \rVert_{L^2(\mu)} \leq \lVert u - KK^* u \rVert_{L^2(\mu)} \leq \lVert \xi-K\eta \rVert_{L^2(\mu)} \leq C (\lVert \xi \rVert_{C^{1,\sigma}(\X)}+\lVert \eta \rVert_{C^{1,\sigma}(\Y)}),
\end{align*}
where $C>0$ is a constant that is allowed to change from line to line. Since $r_{\pi} \in C^{1,\alpha_0}(\Z)$, $KK^*:L^2(\mu) \to C^{1,\sigma}(\X)$ is bounded. Therefore 
\begin{align*}
    \lVert u \rVert_{C^{1,\sigma}(\X)} &\leq \lVert KK^*u \rVert_{C^{1,\sigma}(\X)}+C (\lVert \xi \rVert_{C^{1,\sigma}(\X)}+\lVert \eta \rVert_{C^{1,\sigma}(\Y)}) \\
    &\leq C \lVert u \rVert_{L^2(\mu)}+C (\lVert \xi \rVert_{C^{1,\sigma}(\X)}+\lVert \eta \rVert_{C^{1,\sigma}(\Y)}) \\
    & \leq C (\lVert \xi \rVert_{C^{1,\sigma}(\X)}+\lVert \eta \rVert_{C^{1,\sigma}(\Y)}),
\end{align*}
which completes the proof. The argument for the boundedness in $L^\infty$ norm is the same. 
\end{proof}

\begin{lemma}\label{lem:fiber-projection}
Let $\mathcal{U}$ be the open neighborhood given in Lemma~\ref{lem:bridge-smooth} when $\mathcal{S}=L^\infty(\X) \times L^\infty(\Y)$, with the bridge defined in~\eqref{eq:bridge-scaling}. For any $(h,k) \in \mathcal{U}$, it holds that
\begin{equation}
\Ent(P\mid\pi)=\Ent(P\mid\Bridge{h\mu}{k\nu})
 +\Ent(\Bridge{h\mu}{k\nu}\mid\pi), \quad \forall P \in \Pi(h\mu,k\nu).
 \label{eq:pythagorean}
\end{equation}
In particular, we have that  $$\Bridge{h\mu}{k\nu}=\arg\min_{P \in \Pi(h\mu,k\nu)} \Ent(P \mid \pi).$$
Furthermore, for the projected Fisher information~\eqref{eq:projected-fisher}, if $(h,k) \in C^1(\X)\times C^1(\Y)$, then for any $P \in \Pi(h\mu,k\nu)$ with $\frac{dP}{d\pi} \in C^1(\Z)$
\begin{equation}
 \PFI(P\mid\pi)=
 \PFI(P\mid\Bridge{h\mu}{k\nu}).
 \label{eq:I-invariance}
\end{equation}
\end{lemma}
\begin{proof}
Denote $\alpha=h\mu$, $\beta=k\nu$. Since $\log\frac{dP}{d\pi}=\log\frac{dP}{d\Bridge{\alpha}{\beta}}+a(x)+b(y)$ and $P,\Bridge{\alpha}{\beta}$ have the same marginals,
\[
\Ent(P\mid\pi)=\Ent(P\mid\Bridge{\alpha}{\beta})+\int(a\oplus b)\,dP=\Ent(P\mid\Bridge{\alpha}{\beta})+\Ent(\Bridge{\alpha}{\beta}\mid\pi),
\]
and this equality still holds when $\Ent(P\mid\pi)=+\infty$. Moreover, $\nabla_x\log\frac{dP}{d\pi}$ and $\nabla_x\log\frac{dP}{d\Bridge{\alpha}{\beta}}$ differ by $\nabla_xa(x)$, which is $\sigma(X)$-measurable and thus cancels after subtracting $\E^P[\,\cdot\mid X]$. This yields~\eqref{eq:I-invariance}.
\end{proof}

For stability of Schr\"odinger potentials and optimal couplings under
perturbations of the marginals, see \cite{NutzWiesel2023}.
The following proposition gives the local relative-entropy estimate needed here.

\begin{proposition}
\label{prop:bridge-stability}
There exist $\eta>0$ and $C<\infty$ such that, whenever
$\alpha=h\mu$ and $\beta=k\nu$ are probability measures with
\begin{equation}
 h \in L^\infty(\X), \qquad k \in L^\infty(\Y), \qquad \|h-1\|_\infty\le\eta,
 \qquad
 \|k-1\|_\infty\le\eta,
 \label{eq:bridge-neighborhood}
\end{equation}
the coupling $\Bridge{\alpha}{\beta}$ is well-defined and satisfies
\[
\Ent(\Bridge{\alpha}{\beta}\mid\pi)
 \le C
\bigl(\Ent(\alpha\mid\mu)+\Ent(\beta\mid\nu)\bigr).
\]
\end{proposition}

\begin{proof}
Thanks to Lemma~\ref{lem:bridge-smooth}, there exists $\eta>0$ small enough so that $ \Bridge{\alpha}{\beta}$ is well-defined. With $r_*$ from~\eqref{eq:bridge-reference-density}, we choose $\eta<r_*/4$. For any $h,k$ satisfying \eqref{eq:bridge-neighborhood}, set $u=h-1$, $v=k-1$, and
\begin{align*}
    d\tilde P(x,y)=\left(r_\pi(x,y)+u(x)+v(y) \right) \, d\mu(x) \, d\nu(y).
\end{align*}
Clearly, $\tilde P$ is a nonnegative measure on $\Z$. For any $f \in C_b(\X)$, 
\begin{align*}
    \int_{\Z} f(x) \, d \tilde P (x,y)&= \int_\X f(x) \, d\mu(x)+ \int_\X f(x) u(x) \, d\mu(dx)  =\int_\X f(x) \, d \alpha(x),
\end{align*}
where we used $\pi \in \Pi(\mu,\nu)$ and $\int_\Y v(y) \, d\nu(y)=0$. Hence, the $\X$-marginal of $\tilde P$ is $\alpha$. Similarly, its $\Y$-marginal is $\beta$, i.e. $\tilde P \in \Pi(\alpha,\beta)$. 

By Taylor expansion, it can be verified that for any $\delta \in (0,1)$
\begin{align}\label{eq:entropy-taylor-bounds}
    \frac{z^2}{2(1+\delta)} \leq (1+z)\log(1+z)-z \leq \frac{z^2}{2(1-\delta)}, \qquad |z|\leq \delta. 
\end{align}
Note that $\frac{d\tilde P}{d\pi}=1+\frac{u\oplus v}{r_{\pi}}$ with  $\frac{u\oplus v}{r_{\pi}} \in [-1/2,1/2]$. Applying \eqref{eq:entropy-taylor-bounds} with $z=\frac{u\oplus v}{r_{\pi}}$ $\delta=1/2$,
\begin{align*}
    \Ent(\tilde P \mid \pi)&= \int_{\Z}(1+z) \log (1+z) \, d\pi= \int_\Z (1+z) \log(1+z) -z \, d\pi \leq \int_\Z z^2 \, d \pi  \\
    &=\int_\Z \frac{(u(x)+v(y))^2}{r_\pi(x,y)} \, d\mu(x) \, d\nu(y)\leq 2 r_*^{-1}  (\lVert u \rVert_{L^2(\mu)}^2+\lVert v \rVert_{L^2(\nu)}^2).
\end{align*}
On the other hand, applying \eqref{eq:entropy-taylor-bounds} with $\delta=\eta$, we obtain that 
\begin{align*}
    \Ent(\alpha \mid \mu)&= \int_\X (1+u) \log(1+u) -u \, d\mu \geq \frac{1}{2(1+\eta)} \lVert u \rVert_{L^2(\mu)}^2, \\
     \Ent(\beta \mid \nu)&= \int_\Y (1+v) \log(1+v) -v \, d\nu \geq \frac{1}{2(1+\eta)} \lVert v \rVert_{L^2(\nu)}^2.
\end{align*}
Choosing $C=4(1+\eta)/r_*$, we complete the proof. 
\end{proof}

\subsection{A projected log-Sobolev inequality}

\begin{proposition}\label{prop:global-plsi}
Fix $\underline q,\overline q,L>0$. There exists $C<\infty$, depending only on $\underline q,\overline q,L$ and the torus, such that for every $Q\in\Pc(\Z)$ with density $q$ satisfying
\begin{equation}\label{eq:plsi-reference-class}
q\in C^1(\Z),\qquad \underline q\le q\le\overline q,\qquad \|\nabla\log q\|_\infty\le L,
\end{equation}
we have
\begin{equation}\label{eq:global-plsi}
\Ent(P\mid Q)\le C\,\PFI(P\mid Q)\qquad\text{for all } P\in\Pi(Q_\X,Q_\Y).
\end{equation}
\end{proposition}

\begin{proof}
\noindent\emph{Step 1.} Write the conditional logarithmic gradients as
\begin{equation}\label{eq:conditionalmean}
m:=\CondScore{P}{Q}=(m_\X(x),m_\Y(y)),\qquad m_\X(x):=\E^P\big[\nabla_x\log\tfrac{dP}{dQ}\,\big|\,X=x\big],
\end{equation}
and symmetrically for $m_\Y$. Set
\begin{equation}\label{eq:plsi-density-hellinger}
u:=\sqrt{dP/dQ},\qquad \Hell^2(P,Q):=\int(u-1)^2\,dQ.
\end{equation}
 Since $P$ and $Q$ have the same marginals,
\begin{equation}\label{eq:plsi-marginal-constraints}
\E^{Q}[u^2\vert X]=\E^{Q}[u^2\vert Y]=1;
\end{equation}
in particular $\int u^2\,dQ=1$ and $0\le\Hell^2(P,Q)\le2$. A direct computation gives $\PFI(P\mid Q)=\int|2\nabla u-um|^2\,dQ$. Since $\int(2\nabla u-um)\cdot um\,dQ=0$ and $\int|um|^2\,dQ=\int|m|^2\,dP=\int|m|^2\,dQ$, we get
\begin{equation}\label{eq:plsi-sqrt-identities}
4\int|\nabla u|^2\,dQ=\PFI(P\mid Q)+\|m\|_{L^2(Q)}^2.
\end{equation}

\smallskip
\noindent\emph{Step 2.} We provide a uniform bound for $m$. Let $q_\X(x)=\int_\Y q(x,y)\,dy$. The identity
$\int_\Y u(x,y)^2q(x,y)\,dy=\allowbreak q_\X(x)$ can be differentiated weakly.
Indeed, $u\in H^1(Q)$ implies $u^2q\in W^{1,1}(\Z)$, and integration in $y$
commutes with weak differentiation. Thus
\[
 \nabla_x q_\X(x)
 =\int_\Y\bigl(2u\nabla_xu\,q+u^2\nabla_xq\bigr)(x,y)\,dy.
\]
Since $P,Q$ have the same marginals, and $u^2=\frac{dP}{dQ}$, it can be directly computed that 
\begin{align*}
    m_{\X}(x)=\frac{1}{q_\X(x)}\int_\Y 2u(x,y)\nabla_xu(x,y)q(x,y)\,dy,
\end{align*}
and hence 
\begin{align*}
     \nabla_x\log q_\X(x)
 &=m_{\X}(x)
   +\frac{1}{q_\X(x)}\int_\Y u(x,y)^2\nabla_xq(x,y)\,dy
   =m_{\X}(x)+\E^{P}[\nabla_x\log q\vert X=x].
\end{align*}
Together with the equality $\nabla_x \log q_{\X}(x)= \E^{Q}[\nabla_x \log q \vert X=x]$, $\E^{P}[\nabla_x \log q \vert X=x]=\allowbreak\E^{Q} [u^2 \nabla_x \log q \vert X=x]$, we obtain 
\begin{equation*}
 m_{\X}
 =\nabla_x\log q_\X-\E^{P}[\nabla_x\log q\vert X]
 =\E^{Q}[(1-u^2)\nabla_x\log q\vert X],
\end{equation*}
and similarly $
    m_{\Y}=\E^{Q}[(1-u^2) \nabla_y \log q \vert Y]$. 
It follows from $\|\nabla\log q\|_\infty\le L$ that $\|m_{\X}\|_{\infty}\le2L$ and $\|m_{\Y}\|_{\infty}\le2L$.
Moreover, conditional Cauchy--Schwarz and
\eqref{eq:plsi-marginal-constraints} yield
$$
 |m_{\X}|^2
 \le L^2\E^{Q}[(u-1)^2\vert X]\,
              \E^{Q}[(u+1)^2\vert X]
              \le4L^2\E^{Q}[(u-1)^2\vert X],
$$
where we use $\E^{Q}[u\vert X]\le \sqrt{\E^{Q}[u^2\vert X]}=1$. Integrating against $Q_\X$ and repeating
the argument for $m_{\Y}$, with $\Hell$ as in~\eqref{eq:plsi-density-hellinger}, we get that
\begin{equation}\label{eq:plsi-score-bounds}
 \|m\|_\infty\le 2\sqrt{2}L,\qquad
 \|m\|_{L^2(Q)}^2\le 8L^2 \Hell^2(P,Q).
\end{equation}

\smallskip
\noindent\emph{Step 3.}
We establish a uniform  estimate:
\begin{equation}\label{eq:plsi-hellinger-gap}
 \Hell^2(P,Q)\le C_H\,\PFI(P\mid Q),
\end{equation}
where $C_H$ depends only on $\underline q,\overline q,L$ and the torus.
Suppose that no such constant exists. 
Take a sequence of $Q_n \in \Pc(\Z)$ satisfying~\eqref{eq:plsi-reference-class} and $P_n \in \Pi((Q_n)_\X,(Q_n)_\Y)$, such that
\begin{equation}
\eta_n:=\Hell(P_n,Q_n)>0,\quad \PFI_n:=\PFI(P_n\mid Q_n)<\infty,
 \quad u_n:=\sqrt{dP_n/dQ_n},\quad \frac{\PFI_n}{\eta_n^2}\to 0.
 \label{eq:plsi-contradiction-sequence}
\end{equation}
Since $\eta_n^2\le2$, we have $\PFI_n\to0$. Let $q_n$ be the density of $Q_n$. By~\eqref{eq:plsi-sqrt-identities}, \eqref{eq:plsi-score-bounds} and $\underline q\le q_n\le\overline q$, $(u_n)$ is bounded in $H^1(\Z)$. Since $|\nabla q_n|\le\overline qL$, the Arzel\`a--Ascoli and Rellich theorems give, along a subsequence, $q_n\to q$ uniformly, $u_n\to u$ in $L^2(\Z)$ and $u_n\rightharpoonup u$ in $H^1(\Z)$. The limit $q$ is a Lipschitz probability density with $\underline q\le q\le\overline q$; set $Q:=q\,dz$. Then $u\ge0$, $\int u^2\,dQ=1$, and since $u_n^2q_n\to u^2q$ in $L^1(\Z)$, $u^2Q$ and $Q$ have the same marginals.

Define $m_n$ similarly as in \eqref{eq:conditionalmean} with $P,Q$ replaced by $P_n,Q_n$. By the uniform $L^\infty$ bound in \eqref{eq:plsi-score-bounds}, we extract a convergent subsequence
 $m_n\rightharpoonup^*\widehat m$ in $L^\infty(\Z)$, where $\widehat m_{\X}$ depends
only on $x$ and $\widehat m_{\Y}$ only on $y$. These coordinate subspaces are
preserved under the indicated weak limits. Since $\PFI_n \to 0$, 
$2\nabla u_n-u_nm_n\to0$ in $L^2(\Z)$.
According to the strong convergence of $u_n$, and the uniform bound on $m_n$ \eqref{eq:plsi-score-bounds}, we have that $u_nm_n\rightharpoonup u\widehat m$ in $L^2(\Z)$, and hence
\begin{equation}\label{eq:plsi-nonlinear-zero}
 2\nabla u=u\widehat m.
\end{equation}

We next show that $u$ is bounded from above and below away from zero.
For each $\delta>0$, due to chain rule,
\(
\left|\nabla\log(u+\delta)\right|
 =\frac{u|\widehat m|}{2(u+\delta)}\le \sqrt{2}L.
\)
The function $\log(u+\delta)$ is in $L^2(\Z)$ and has bounded weak gradient.
It therefore has a Lipschitz version on the torus. Connectedness
and compactness imply
\[
 \operatorname*{ess\,sup}(u+\delta)
 \le C_{\mathrm{Har}}\operatorname*{ess\,inf}(u+\delta),
 \qquad
 C_{\mathrm{Har}}=\exp\!\left(\sqrt{2}L \operatorname{diam}(\Z)\right),
\]
independently of $\delta$. Letting $\delta\to 0$, and using
$\int u^2\,d Q=1$, gives $C_{\mathrm{Har}}^{-1}\le u\le C_{\mathrm{Har}}$ almost everywhere.
Hence $\log u^2\in W^{1,\infty}$, and \eqref{eq:plsi-nonlinear-zero} yields
\(
 \nabla_x\log u^2=\widehat m_{\X}(x),\ 
 \nabla_y\log u^2=\widehat m_{\Y}(y).
\)
It follows that $\log u^2=A(x)+D(y)$ for some functions $A \in L^\infty(\X)$ and $D \in L^\infty(\Y)$.

Since $u^2Q$ and $Q$ have the same marginals, and $A,D$ are bounded,
\[
 \int (u^2-1)\log u^2\,d Q
 =\int (u^2-1)(A(x)+D(y))\,d Q=0.
\]
The integrand on the left is nonnegative and vanishes only when $u=1$.
Thus $u=1$ almost everywhere. Uniform convergence of $q_n$ and strong
convergence of $u_n$ now imply
\begin{equation}\label{eq:plsi-hellinger-to-zero}
 \eta_n^2=\int_{\Z} (u_n-1)^2 \, dQ_n\to0.
\end{equation}

To derive a contradiction, with $\eta_n$ as in~\eqref{eq:plsi-contradiction-sequence}, let us take
\(
 v_n=\frac{u_n-1}{\eta_n}.
\)
Then $\int v_n^2\,d Q_n=1$, while according to
\eqref{eq:plsi-sqrt-identities}, \eqref{eq:plsi-score-bounds},
\[
 4\int|\nabla v_n|^2\,d Q_n
 \le\frac{\PFI_n}{\eta_n^2}+8L^2,
 \qquad
 \left\|\frac{m_n}{\eta_n}\right\|_{L^2(Q_n)}\le 2\sqrt{2}L.
\]
Thus $(v_n)$ is bounded in $H^1$. After a further extraction,
\[
 v_n\to v\ \text{strongly in }L^2(\Z),\qquad
 v_n\rightharpoonup v\ \text{weakly in }H^1(\Z),\qquad
 \int v^2\,d Q=1.
\]
As $\E^{Q_n}[u_n^2 \vert X]=1$ and $u_n=\eta_nv_n+1$, we have
\(
 2\E^{Q_n}[v_n\vert X]+\eta_n\E^{Q_n}[v_n^2\vert X]=0,
\)
and similarly for the conditional expectation w.r.t. $Y$. Since $\eta_n \to 0$ in \eqref{eq:plsi-hellinger-to-zero},
$v_n \to v$ in $L^2(\Z)$, and $q_n \to q$ uniformly, letting $n\to \infty$ in the equality above  we get that 
\begin{equation}\label{eq:plsi-limit-constraints}
 \E^{Q}[v\vert X]=\E^{Q}[v\vert Y]=0.
\end{equation}

As $\frac{\PFI_n}{\eta_n^2} \to 0$ and $q_n \geq \underline q$, by direct computation
\(
 2\nabla v_n-\frac{m_n}{\eta_n}-v_nm_n
 \to 0 \ \text{in }L^2(\Z).
\)
The product term converges to zero in $L^1(\Z)$, 
\[
 \int|v_nm_n|\,dz
 \le\frac{1}{\underline q}\|v_n\|_{L^2(Q_n)}\|m_n\|_{L^2(Q_n)}
 \le\frac{2\sqrt{2}L}{\underline q} \eta_n\to 0.\]
Therefore $m_n/\eta_n$ is bounded in $L^2(\Z)$. Extract a weak $L^2(\Z)$ limit $\tilde m$ of $m_n/\eta_n$. It can be easily verified that $\tilde m_{\X}$ and $\tilde m_{\Y}$ depend only on $x$ and $y$ respectively, and the same for $\nabla_xv$ and $\nabla_yv$. Hence  $v=A_v(x)+D_v(y)$, with
$A_v\in H^1(\X)$ and $D_v\in H^1(\Y)$.
Using \eqref{eq:plsi-limit-constraints}, we obtain
\[
 \int v^2\,d Q
 =\int vA_v(x)\,d Q+\int vD_v(y)\,d Q=0.
\]
This contradicts $\int v^2\,d Q=1$ and proves
\eqref{eq:plsi-hellinger-gap}.

\smallskip
\noindent\emph{Step 4.} By~\eqref{eq:relative-entropy} and~\eqref{eq:plsi-density-hellinger}, the logarithmic Sobolev inequality on the torus gives $\Ent(P\mid Q)\le 4C_{\mathrm{LS}}\int|\nabla u|^2\,dQ$, with $C_{\mathrm{LS}}$ depending only on $\underline q,\overline q$. Combined with~\eqref{eq:plsi-sqrt-identities}, \eqref{eq:plsi-score-bounds} and~\eqref{eq:plsi-hellinger-gap}, this yields
\[
\Ent(P\mid Q)\le C_{\mathrm{LS}}\bigl(\PFI(P\mid Q)+8L^2\Hell^2(P,Q)\bigr)\le C_{\mathrm{LS}}(1+8L^2C_H)\,\PFI(P\mid Q),
\]
which is~\eqref{eq:global-plsi} with $C=C_{\mathrm{LS}}(1+8L^2C_H)$.
\end{proof}

The following result is a direct consequence of~\eqref{eq:global-plsi}.
\begin{corollary}\label{cor:nearby-plsi}
There exist $\eta\in(0,1)$ and $\rho>0$, depending only on $\pi$ and
the fixed torus, such that for every probability measure $Q \in \Pc(\Z)$ satisfying
\(
 \left\|\frac{dQ}{d\pi}-1\right\|_{C^1}\le\eta,
\)
it holds that 
\[
 \Ent(P\mid Q)\le\frac1{2\rho}\PFI(P\mid Q),
 \qquad \forall \, P\in\Pi(Q_\X,Q_\Y).
\]
\end{corollary}

\section{Exponential convergence}
\label{sec:convergence-rate}

We prove ~\eqref{eq:main-wasserstein-bound} by comparing $P_t$ with the
entropy-minimizing bridge $\Bridge{\mu_t}{\nu_t}$ having its current
marginals. The decomposition~\eqref{eq:pythagorean} separates the entropy
relative to this bridge from the entropy of the bridge relative to $\pi$.
The latter is controlled by the marginal entropies through
Proposition~\ref{prop:bridge-stability}; the former satisfies the
dissipation identity proved below. Throughout this section,
Assumption~\ref{ass:regularity} holds. Constants denoted by $C$ depend only
on the fixed model data and may change from one occurrence to the next.

For probability measures with positive densities and square-integrable
logarithmic gradient, the ordinary relative Fisher information is
\[
 \FI(P\mid Q):=\int\left|\nabla\log\frac{dP}{dQ}\right|^2\,dP.
\]
We retain $\PFI(P\mid Q)$ for the projected Fisher information
defined in~\eqref{eq:projected-fisher}.

\subsection{Entropy relative to the moving bridge}
\label{sec:moving-bridge-entropy}

\begin{lemma}\label{lem:bridge-entropy-evolution}
Fix $0<\sigma<\alpha_0$ and write $\mathcal B_t:=\Bridge{\mu_t}{\nu_t}$. On any positive-time interval on which the marginal density ratios remain in the $C^{1,\sigma}$ neighborhood of Lemma~\ref{lem:bridge-smooth}, $t\mapsto\Ent(P_t\mid\mathcal B_t)$ is locally absolutely continuous and, for a.e.\ $t$ in that interval,
\begin{align}
&\frac{d}{dt}\Ent(P_t\mid\mathcal B_t)
={}-\epsilon\PFI(P_t\mid\mathcal B_t)\notag-\int_\X\nabla_x\log\tfrac{d\mu_t}{d\mu}\cdot\bigl(\E^{P_t}-\E^{\mathcal B_t}\bigr)[\nabla_xc\mid X=x]\,d\mu_t(x)\notag\\
&-\int_\Y\nabla_y\log\tfrac{d\nu_t}{d\nu}\cdot\bigl(\E^{P_t}-\E^{\mathcal B_t}\bigr)[\nabla_yc\mid Y=y]\,d\nu_t(y).
\label{eq:bridge-rate-exact-identity}
\end{align}
Consequently,
\begin{equation}\label{eq:bridge-rate-bridge-differential}
\frac{d}{dt}\Ent(P_t\mid\mathcal B_t)\le-\epsilon\PFI(P_t\mid\mathcal B_t)
+\sqrt{2\bigl(\|\nabla_xc\|_\infty^2+\|\nabla_yc\|_\infty^2\bigr)\Ent(P_t\mid\mathcal B_t)\bigl[\FI(\mu_t\mid\mu)+\FI(\nu_t\mid\nu)\bigr]}.
\end{equation}
\end{lemma}

\begin{proof}
\emph{Step 1: Differentiation with a moving reference.} Since
both $P_t$ and the reference measure
$\Bridge{\mu_t}{\nu_t}$ depend on time.
To deal with the time dependence of the bridge, we express
this entropy as a difference of two entropies relative to
the fixed measure $\pi$.
By definition, $\Bridge{\mu_t}{\nu_t}$ has marginals
$\mu_t,\nu_t$, which are also the marginals of $P_t$.
Moreover, the bridge representation yields
\(
 \log\frac{d\Bridge{\mu_t}{\nu_t}}{d\pi}(x,y)
 =a_t(x)+b_t(y).
\)
Since measures with the same marginals integrate every
function of this additive form equally, we have
\(
 \int \log\frac{d\Bridge{\mu_t}{\nu_t}}{d\pi}\,dP_t
 =
 \int \log\frac{d\Bridge{\mu_t}{\nu_t}}{d\pi}
       \,d\Bridge{\mu_t}{\nu_t}.
\)
Thus, as in Lemma~\ref{lem:fiber-projection},
\begin{equation}
\Ent(P_t\mid\pi)=\Ent(P_t\mid\Bridge{\mu_t}{\nu_t})+\Ent(\Bridge{\mu_t}{\nu_t}\mid\pi).
 \label{eq:bridge-rate-decomposition}
\end{equation}
For any two times $t,u$ in the bridge interval, we have
\begin{align*}
 &\Ent(\Bridge{\mu_u}{\nu_u}\mid\pi)
 -\Ent(\Bridge{\mu_t}{\nu_t}\mid\pi)=
 \Ent(\Bridge{\mu_u}{\nu_u}\mid\Bridge{\mu_t}{\nu_t})
 +\int\log\frac{d\Bridge{\mu_t}{\nu_t}}{d\pi}
       \,d(\Bridge{\mu_u}{\nu_u}-\Bridge{\mu_t}{\nu_t}).
\end{align*}
For each $s\in\{t,u\}$, the measures $P_s$ and
$\Bridge{\mu_s}{\nu_s}$ have the same marginals.
They therefore integrate every function of the form
$a(x)+b(y)$ equally. Since
$\log(d\Bridge{\mu_t}{\nu_t}/d\pi)$ has this form,
the last integral equals
\(
\int\log\frac{d\Bridge{\mu_t}{\nu_t}}{d\pi}\,d(P_u-P_t).
\)
Nonnegativity of
$\Ent(\Bridge{\mu_u}{\nu_u}\mid\Bridge{\mu_t}{\nu_t})$
then gives a lower bound for the entropy difference.
Interchanging $t$ and $u$ gives the corresponding upper bound.
Together, these yield
\begin{align}
 \int\log\frac{d\Bridge{\mu_t}{\nu_t}}{d\pi}\,d(P_u-P_t)
 &\le
 \Ent(\Bridge{\mu_u}{\nu_u}\mid\pi)
 -\Ent(\Bridge{\mu_t}{\nu_t}\mid\pi)\notag\\
 &\le
 \int\log\frac{d\Bridge{\mu_u}{\nu_u}}{d\pi}\,d(P_u-P_t).
 \label{eq:bridge-rate-sandwich}
\end{align}
Let $\sigma<\theta<\alpha_0$, by 
Lemma~\ref{lem:asymptotic-compactness} $p_s$ is continuous and of the marginal density ratios in $C^{1,\sigma}$. Thus
$g_s:=\log(d\Bridge{\mu_s}{\nu_s}/d\pi)$ varies continuously in $C^1(\Z)$
by Lemma~\ref{lem:bridge-smooth}. The formula for the conditional drift
and positivity of the marginal densities give continuity in $L^2(\Z)$ of
the flux $j_s:=p_s\Drift[P_s]-\epsilon\nabla p_s$.
For each fixed $f\in C^1(\Z)$ the weak Fokker--Planck equation yields that 
\(
 \int_\Z f\,d(P_u-P_t)=\int_t^u\!\int_\Z \nabla f\cdot j_s\,dz\,ds.
\)
For $f=g_t$ and $f=g_u$, the absolute value of this integral is at most
$C|u-t|$. Hence \eqref{eq:bridge-rate-sandwich} gives local Lipschitz
continuity of the bridge entropy. Dividing both bounding integrals by
$u-t$ and letting $u\to t$ gives the same limit on each side, so
\begin{equation}
 \frac{d}{dt}\Ent(\Bridge{\mu_t}{\nu_t}\mid\pi)
 =\int_\Z \nabla\log\frac{d\Bridge{\mu_t}{\nu_t}}{d\pi}\cdot j_t\,dz.
 \label{eq:bridge-rate-bridge-derivative}
\end{equation}
For the fixed reference $\pi$, apply Lemma~\ref{lem:weak-chain-rule}
with $q^1_s=p_s$, $b^1_s=-\Drift[P_s]$, $q^2_s=\pi$ and
$b^2_s=-\epsilon\nabla\log\pi$. For almost every $t$ this gives
\[
 \frac{d}{dt}\Ent(P_t\mid\pi)
 =\int_\Z \nabla\log\frac{dP_t}{d\pi}\cdot j_t\,dz.
\]
Differentiating \eqref{eq:bridge-rate-decomposition} at these times and
subtracting \eqref{eq:bridge-rate-bridge-derivative} yields
\begin{equation}
 \frac{d}{dt}\Ent(P_t\mid\Bridge{\mu_t}{\nu_t})
 =\int_\Z \nabla\log\frac{dP_t}{d\Bridge{\mu_t}{\nu_t}}\cdot j_t\,dz.
 \label{eq:bridge-rate-moving-chain}
\end{equation}
In particular, this argument does not assume time differentiability of
the bridge potentials.

\medskip
\noindent \emph{Step 2: Conditional logarithmic gradients and the flux.}
Differentiating the marginal density with respect to $x$ gives
\begin{align}
 \E^{P_t}[\nabla_x\log p_t\mid X=x]
 &=\frac1{\mu_t(x)}\int_\Y\nabla_xp_t(x,y)\,d y
 =\nabla_x\log\mu_t(x).
 \label{eq:bridge-rate-marginal-score}
\end{align}
The same calculation for the bridge gives the same right-hand side.
Moreover, the Gibbs form implies that
$\nabla_x\log\Bridge{\mu_t}{\nu_t}+\epsilon^{-1}\nabla_xc$ depends only on $x$.
Its conditional expectation given $X$ is therefore itself. Thus
\begin{equation} \label{eq:bridge-rate-centered-gibbs}
 \nabla_x\log\Bridge{\mu_t}{\nu_t}
       -\E^{P_t}[\nabla_x\log\Bridge{\mu_t}{\nu_t}\mid X]
       =-\frac1\epsilon\bigl(\nabla_xc-\E^{P_t}[\nabla_xc\mid X]\bigr).
\end{equation}
Expanding the logarithmic ratio and using
\eqref{eq:bridge-rate-marginal-score}--\eqref{eq:bridge-rate-centered-gibbs},
\begin{equation}\label{eq:bridge-rate-expanded-projection}
 \begin{aligned}
 &\nabla_x\log\frac{dP_t}{d\Bridge{\mu_t}{\nu_t}}
 -\E^{P_t}\!\left[\nabla_x\log\frac{dP_t}{d\Bridge{\mu_t}{\nu_t}}\mid X\right]\\
 &\qquad=
 \nabla_x\log p_t-\nabla_x\log\mu_t
 +\frac1\epsilon\bigl(\nabla_xc-\E^{P_t}[\nabla_xc\mid X]\bigr).
 \end{aligned}
\end{equation}
Since $\nabla_x\log(d\mu_t/d\mu)
=\nabla_x\log\mu_t+\nabla U$, adding this gradient to
\eqref{eq:bridge-rate-expanded-projection} cancels the marginal logarithmic gradient.
Multiplication by $\epsilon p_t$ and \eqref{eq:drift-definition} give
\begin{align}
 \epsilon\nabla_xp_t-p_t\Drift_\X[P_t]
 =\epsilon p_t\Biggl(&\nabla_x\log\frac{dP_t}{d\Bridge{\mu_t}{\nu_t}}
 -\E^{P_t}\!\left[\nabla_x\log\frac{dP_t}{d\Bridge{\mu_t}{\nu_t}}\mid X\right]
 +\nabla_x\log\frac{d\mu_t}{d\mu}\Biggr).
 \label{eq:bridge-rate-flux}
\end{align}
The same calculation holds in the $y$-coordinate.
Write $R_t:=\log\frac{dP_t}{d\mathcal B_t}$, with $\mathcal B_t=\Bridge{\mu_t}{\nu_t}$ as in Lemma~\ref{lem:bridge-entropy-evolution}. Shared marginals and the Gibbs form also give
\begin{equation}\label{eq:bridge-rate-conditional-force}
\epsilon\,\E^{P_t}[\nabla_xR_t\mid X=x]
=\epsilon\bigl(\E^{\mathcal B_t}-\E^{P_t}\bigr)[\nabla_x\log\mathcal B_t\mid X=x]
=\bigl(\E^{P_t}-\E^{\mathcal B_t}\bigr)[\nabla_xc\mid X=x],
\end{equation}
and analogously in $y$. Combining~\eqref{eq:bridge-rate-flux} with its $y$-analogue, the Fokker--Planck equation reads
\begin{equation}\label{eq:bridge-rate-fokker-planck}
\begin{aligned}
\partial_tp_t={}&\epsilon\nabla_x\cdot\bigl[p_t\bigl(\nabla_xR_t-\E^{P_t}[\nabla_xR_t\mid X]+\nabla_x\log\tfrac{d\mu_t}{d\mu}\bigr)\bigr]\\
&+\epsilon\nabla_y\cdot\bigl[p_t\bigl(\nabla_yR_t-\E^{P_t}[\nabla_yR_t\mid Y]+\nabla_y\log\tfrac{d\nu_t}{d\nu}\bigr)\bigr].
\end{aligned}
\end{equation}
\medskip
\noindent\emph{Step 3. Conditional orthogonality.}
Using the $x$- and $y$-components of the flux in
\eqref{eq:bridge-rate-fokker-planck}, we can write
\eqref{eq:bridge-rate-moving-chain} as
\begin{align}
 \frac{d}{dt}\Ent(P_t\mid\Bridge{\mu_t}{\nu_t})
 ={}&-\int_\Z\nabla_x\log\frac{dP_t}{d\Bridge{\mu_t}{\nu_t}}\cdot
       (\epsilon\nabla_xp_t-p_t\Drift_\X[P_t])\,d x\,d y\notag\\
 &-\int_\Z\nabla_y\log\frac{dP_t}{d\Bridge{\mu_t}{\nu_t}}\cdot
       (\epsilon\nabla_yp_t-p_t\Drift_\Y[P_t])\,d x\,d y.
 \label{eq:bridge-rate-ibp}
\end{align}
For every square-integrable vector field $f$, conditioning on $X$ shows that $\E^{P_t}[f\mid X]\cdot(f-\E^{P_t}[f\mid X])$ has zero mean, hence $\E^{P_t}[f\cdot(f-\E^{P_t}[f\mid X])]=\E^{P_t}[|f-\E^{P_t}[f\mid X]|^2]$. Inserting~\eqref{eq:bridge-rate-flux} into~\eqref{eq:bridge-rate-ibp} and applying this identity, the two centered terms sum to $-\epsilon\PFI(P_t\mid\mathcal B_t)$ by~\eqref{eq:projected-fisher}. By~\eqref{eq:bridge-rate-conditional-force}, the remaining $x$-term is
\begin{align*}
-\epsilon\int_\Z\nabla_x\log\tfrac{d\mu_t}{d\mu}\cdot\nabla_xR_t\,dP_t
&=-\epsilon\int_\X\nabla_x\log\tfrac{d\mu_t}{d\mu}\cdot\E^{P_t}[\nabla_xR_t\mid X=x]\,d\mu_t(x)\\
&=-\int_\X\nabla_x\log\tfrac{d\mu_t}{d\mu}\cdot\bigl(\E^{P_t}-\E^{\mathcal B_t}\bigr)[\nabla_xc\mid X=x]\,d\mu_t(x),
\end{align*}
and the corresponding $y$-calculation completes the proof of~\eqref{eq:bridge-rate-exact-identity}.

\medskip
\noindent\emph{Step 4: Conditional Pinsker and Cauchy--Schwarz.}
For bounded vector-valued $f$, Pinsker's inequality gives
\[
 \left|\int f\,d\alpha-\int f\,d\beta\right|^2
 \le 2\|f\|_\infty^2\Ent(\alpha\mid\beta).
\]
Applying this to the conditional laws given $X=x$ yields
$$
 \left|\E^{P_t}[\nabla_xc\mid X=x]
        -\E^{\Bridge{\mu_t}{\nu_t}}[\nabla_xc\mid X=x]\right|^2
        \le 2\|\nabla_xc\|_\infty^2
 \Ent\bigl((P_t)_{\Y\mid x}\mid(\Bridge{\mu_t}{\nu_t})_{\Y\mid x}\bigr).
$$
Because the $X$-marginals agree, the entropy chain rule gives
\[
 \int_\X \Ent\bigl((P_t)_{\Y\mid x}\mid(\Bridge{\mu_t}{\nu_t})_{\Y\mid x}\bigr)\,d\mu_t(x)
 =\Ent(P_t\mid\Bridge{\mu_t}{\nu_t}).
\]
Consequently,
\[
 \int_\X\left|\E^{P_t}[\nabla_xc\mid X=x]
          -\E^{\Bridge{\mu_t}{\nu_t}}[\nabla_xc\mid X=x]\right|^2\,d\mu_t(x)
 \le 2\|\nabla_xc\|_\infty^2\Ent(P_t\mid\Bridge{\mu_t}{\nu_t}).
\]
The analogous estimate holds in $y$. Cauchy--Schwarz bounds the sum of
the absolute values of the two error terms in \eqref{eq:bridge-rate-exact-identity} by
\begin{align*}
 &\sqrt{2\Ent(P_t\mid\Bridge{\mu_t}{\nu_t})}
   \left(\|\nabla_xc\|_\infty\sqrt{\FI(\mu_t\mid\mu)}
        +\|\nabla_yc\|_\infty\sqrt{\FI(\nu_t\mid\nu)}\right)\\
 &\quad\le
 \sqrt{2\bigl(\|\nabla_xc\|_\infty^2+\|\nabla_yc\|_\infty^2\bigr)}
 \sqrt{\Ent(P_t\mid\Bridge{\mu_t}{\nu_t})
       \bigl[\FI(\mu_t\mid\mu)+\FI(\nu_t\mid\nu)\bigr]}.
\end{align*}
This proves \eqref{eq:bridge-rate-bridge-differential}.
\end{proof}

\subsection{Proof of the convergence rate}
\label{sec:proof-convergence-rate}

\begin{proof}[Proof of Theorem~\ref{thm:main}]
Existence and uniqueness follow from Proposition~\ref{prop:well-posedness}.
It remains to prove~\eqref{eq:main-wasserstein-bound}. Throughout this proof,
$\rho>0$ is the constant in Corollary~\ref{cor:nearby-plsi}.

\smallskip
\noindent\emph{Step 1. Marginal entropy estimates.}
Integrating the joint Fokker--Planck equation over one coordinate cancels
the conditional interaction drift, so the marginals solve
\[
 \partial_t\mu_t=\epsilon\nabla_x\cdot
      \left(\mu_t\nabla_x\log\frac{d\mu_t}{d\mu}\right),\qquad
 \partial_t\nu_t=\epsilon\nabla_y\cdot
      \left(\nu_t\nabla_y\log\frac{d\nu_t}{d\nu}\right).
\]
The entropy chain rule of Lemma~\ref{lem:weak-chain-rule}, applied to the marginals, gives
\begin{equation}
 \frac{d}{dt}
 \bigl[\Ent(\mu_t\mid\mu)+\Ent(\nu_t\mid\nu)\bigr]
 =-\epsilon\bigl[\FI(\mu_t\mid\mu)+\FI(\nu_t\mid\nu)\bigr].
 \label{eq:bridge-rate-marginal-dissipation}
\end{equation}
Since $\mu,\nu$ are positive densities on compact tori, they admit
a common logarithmic Sobolev constant $\lambda>0$; see
\cite{BakryGentilLedoux2014}. With the convention
$\FI(\alpha\mid\mu)\ge2\lambda\Ent(\alpha\mid\mu)$ and likewise for $\nu$,
these inequalities imply
\begin{equation}
 \Ent(\mu_t\mid\mu)+\Ent(\nu_t\mid\nu)
 \le e^{-2\epsilon\lambda(t-s)}
      \bigl[\Ent(\mu_s\mid\mu)+\Ent(\nu_s\mid\nu)\bigr],
 \qquad 0<s\le t.
 \label{eq:bridge-rate-marginal-decay}
\end{equation}
For $0<r\le\epsilon\lambda$, integration by parts in
\eqref{eq:bridge-rate-marginal-dissipation} also gives
\begin{align}
 &\int_s^t e^{-r(t-u)}
       \bigl[\FI(\mu_u\mid\mu)+\FI(\nu_u\mid\nu)\bigr]\,d u\notag\\
 &\quad=\frac1\epsilon e^{-r(t-s)}
       \bigl[\Ent(\mu_s\mid\mu)+\Ent(\nu_s\mid\nu)\bigr]
       -\frac1\epsilon\bigl[\Ent(\mu_t\mid\mu)+\Ent(\nu_t\mid\nu)\bigr]\notag\\
 &\qquad\quad+\frac r\epsilon\int_s^t e^{-r(t-u)}
       \bigl[\Ent(\mu_u\mid\mu)+\Ent(\nu_u\mid\nu)\bigr]\,d u\notag\\
 &\quad\le\frac1\epsilon\left(1+\frac r{2\epsilon\lambda-r}\right)e^{-r(t-s)}
       \bigl[\Ent(\mu_s\mid\mu)+\Ent(\nu_s\mid\nu)\bigr]\notag\\
 &\quad\le\frac2\epsilon e^{-r(t-s)}
       \bigl[\Ent(\mu_s\mid\mu)+\Ent(\nu_s\mid\nu)\bigr].
 \label{eq:bridge-rate-weighted-marginals}
\end{align}
The first inequality uses \eqref{eq:bridge-rate-marginal-decay} and discards the
nonpositive terminal entropy term.

\medskip
\noindent\emph{Step 2. Uniform entry into the bridge neighborhood.} We use time $1$ as a fixed positive starting time;
any fixed positive time would suffice.
Fix $0<\sigma<\theta<\alpha_0$. Lemma~\ref{lem:asymptotic-compactness} and positivity
of $\mu,\nu$ give uniform bounds, over all initial laws, on
\[
 \left\|\frac{d\mu_t}{d\mu}\right\|_{C^{1,\theta}(\X)}
 +\left\|\frac{d\nu_t}{d\nu}\right\|_{C^{1,\theta}(\Y)},\qquad t\ge1.
\]
They also bound $\Ent(\mu_1\mid\mu)+\Ent(\nu_1\mid\nu)$ uniformly.
By \eqref{eq:bridge-rate-marginal-decay} and Pinsker's inequality,
\[
 \left\|\frac{d\mu_t}{d\mu}-1\right\|_{L^1(\mu)}
 +\left\|\frac{d\nu_t}{d\nu}-1\right\|_{L^1(\nu)}
 \le Ce^{-\epsilon\lambda(t-1)},\qquad t\ge1.
\]
We claim that
\begin{equation}
 \left\|\frac{d\mu_t}{d\mu}-1\right\|_{C^{1,\sigma}(\X)}
 +\left\|\frac{d\nu_t}{d\nu}-1\right\|_{C^{1,\sigma}(\Y)}
 \longrightarrow0
 \label{eq:bridge-rate-marginal-entry}
\end{equation}
uniformly over initial laws. We prove this by contradictions. Suppose that the convergence is not uniform over initial laws.
Then there exist $\eta>0$, initial laws $P_0^n$, and times
$t_n\to\infty$ such that the corresponding marginals
$\mu_t^n,\nu_t^n$ satisfy
\(
 \left\|\frac{d\mu_{t_n}^n}{d\mu}-1\right\|_{C^{1,\sigma}(\X)}
 +
 \left\|\frac{d\nu_{t_n}^n}{d\nu}-1\right\|_{C^{1,\sigma}(\Y)}
 \ge\eta.
\)
By the uniform $C^{1,\theta}$ bounds and compact embedding,
a subsequence of these density ratios converges in
$C^{1,\sigma}(\X)\times C^{1,\sigma}(\Y)$.
The preceding $L^1$ estimate, which is uniform over initial
laws, forces both limits to equal $1$.
Thus the left-hand side converges to zero along this
subsequence, contradicting the lower bound $\eta$.

Lemma~\ref{lem:bridge-smooth}, Proposition~\ref{prop:bridge-stability}, and
Corollary~\ref{cor:nearby-plsi} therefore
apply after a fixed time $T\ge1$, independent of $P_0$. For all $t\ge T$,
\begin{align}
 \Ent(P_t\mid\pi)
 &=\Ent(P_t\mid\Bridge{\mu_t}{\nu_t})+\Ent(\Bridge{\mu_t}{\nu_t}\mid\pi),
 \label{eq:bridge-rate-late-decomposition}\\
 \Ent(\Bridge{\mu_t}{\nu_t}\mid\pi)
 &\le C\bigl[\Ent(\mu_t\mid\mu)+\Ent(\nu_t\mid\nu)\bigr],
 \label{eq:bridge-rate-bridge-bound}\\
 \Ent(P_t\mid\Bridge{\mu_t}{\nu_t})
&\le\frac1{2\rho}\PFI(P_t\mid\Bridge{\mu_t}{\nu_t}).
 \label{eq:bridge-rate-bridge-lsi}
\end{align}

\medskip
\noindent\emph{Step 3. Decay of the entropy relative to the bridge.}
Combining Lemma~\ref{lem:bridge-entropy-evolution} with \eqref{eq:bridge-rate-bridge-lsi}, and applying
Young's inequality, gives
\begin{align}
 \frac{d}{dt}\Ent(P_t\mid\Bridge{\mu_t}{\nu_t})
 \le{}&-\epsilon\rho \Ent(P_t\mid\Bridge{\mu_t}{\nu_t})+\frac{\|\nabla_xc\|_\infty^2+\|\nabla_yc\|_\infty^2}{2\epsilon\rho}
       \bigl[\FI(\mu_t\mid\mu)+\FI(\nu_t\mid\nu)\bigr].
 \label{eq:bridge-rate-bridge-young}
\end{align}
Indeed, the mixed term in \eqref{eq:bridge-rate-bridge-differential} is bounded by
$\epsilon\rho \Ent(P_t\mid\Bridge{\mu_t}{\nu_t})$ plus the last term in
\eqref{eq:bridge-rate-bridge-young}, whereas the projected dissipation is at most
$-2\epsilon\rho \Ent(P_t\mid\Bridge{\mu_t}{\nu_t})$.

Set
\(
 r:=\epsilon\min\{\rho,\lambda\}>0.
\)
This choice ensures that $\delta\le\tau$ and $\delta<T$.
The integrating-factor inequality
and \eqref{eq:bridge-rate-weighted-marginals} imply, for $t\ge s\ge T$,
\begin{equation}\label{eq:bridge-rate-bridge-decay}
 \Ent(P_t \!
 \mid\Bridge{\mu_t}{\nu_t})
 \le e^{-r(t-s)}\Ent(P_s
 \!
 \mid\Bridge{\mu_s}{\nu_s})
+2\frac{\|\nabla c\|_\infty^2}{\epsilon^2\rho}
       e^{-r(t-s)}
       \bigl[\Ent(\mu_s
       \!
       \mid\mu)+\Ent(\nu_s\mid\nu)\bigr].
\end{equation}
By \eqref{eq:bridge-rate-bridge-bound} and \eqref{eq:bridge-rate-marginal-decay},
\(
 \Ent(\Bridge{\mu_t}{\nu_t}\mid\pi)
 \le C e^{-2\epsilon\lambda(t-s)}
       \bigl[\Ent(\mu_s\mid\mu)+\Ent(\nu_s\mid\nu)\bigr].
\)
Adding this to \eqref{eq:bridge-rate-bridge-decay} and using
\eqref{eq:bridge-rate-late-decomposition} yields
\begin{align*}
 \Ent(P_t\mid\pi)
 \le C e^{-r(t-s)}\Bigl[
 \Ent(P_s\mid\Bridge{\mu_s}{\nu_s})+\Ent(\mu_s\mid\mu)+\Ent(\nu_s\mid\nu)\Bigr].
\end{align*}
The decomposition gives $\Ent(P_s\mid\Bridge{\mu_s}{\nu_s})\le \Ent(P_s\mid\pi)$.
The entropy chain rule in each coordinate gives
$\Ent(\mu_s\mid\mu)+\Ent(\nu_s\mid\nu)\le2\Ent(P_s\mid\pi)$.
Consequently,
\begin{equation}
 \Ent(P_t\mid\pi)\le C e^{-r(t-s)}\Ent(P_s\mid\pi),
 \qquad t\ge s\ge T.
 \label{eq:bridge-rate-late-decay}
\end{equation}

\medskip
\noindent\emph{Step 4. The initial interval and arbitrary initial laws.}
The Gibbs form and the marginal identities for $\pi$ imply
$\Drift[\pi]=\epsilon\nabla\log\pi$. By the chain rule of entropy chain,
\begin{align}
 \frac{d}{dt}\Ent(P_t\mid\pi)
 ={}&-\epsilon \FI(P_t\mid\pi)+\int\bigl(\Drift[P_t]-\Drift[\pi]\bigr)\cdot
       \nabla\log\frac{dP_t}{d\pi}\,d P_t.
 \label{eq:bridge-rate-transient-identity}
\end{align}
Conditional Pinsker's inequality and the entropy
chain rule yield that
\begin{equation}
 \int|\Drift[P_t]-\Drift[\pi]|^2\,d P_t
 \le 2\bigl(\|\nabla_xc\|_\infty^2+\|\nabla_yc\|_\infty^2\bigr)
       \Ent(P_t\mid\pi).
 \label{eq:bridge-rate-transient-force}
\end{equation}
Here the integrated conditional entropies are bounded by $\Ent(P_t\mid\pi)$;
equality of marginals is not needed. Completing the square in
\eqref{eq:bridge-rate-transient-identity} and using \eqref{eq:bridge-rate-transient-force} yields
\begin{equation}
 \frac{d}{dt}\Ent(P_t\mid\pi)
 \le \frac{\|\nabla_xc\|_\infty^2+\|\nabla_yc\|_\infty^2}{2\epsilon}
       \Ent(P_t\mid\pi).
 \label{eq:bridge-rate-transient-bound}
\end{equation}
Thus entropy grows by at most a fixed factor over the interval before $T$.

Fix $\tau>0$ and take $\delta=\min\{\tau,1/2\}$.  Starting at a positive time allows us to use the smoothing
estimate even when $\Ent(P_0\mid\pi)$ is infinite. Proposition~\ref{prop:wasserstein-smoothing} gives
\begin{equation}
 \Ent(P_\delta\mid\pi)\le C_\delta \Wass{2}(P_0,\pi)^2.
 \label{eq:bridge-rate-initial-regularization}
\end{equation}
For $t\ge T$, we apply \eqref{eq:bridge-rate-late-decay}
with $s=T$ and control $\Ent(P_T\mid\pi)$ using
\eqref{eq:bridge-rate-transient-bound} on $[\delta,T]$.
For $\delta\le t\le T$, the transient bound gives the same
estimate after enlarging the constant, since
$e^{-r(t-\delta)}\ge e^{-rT}$.
Thus there is a constant $C_*$, independent of $P_0$ and
$\tau$, such that
\[
 \Ent(P_t\mid\pi)\le C_* e^{-r(t-\delta)}\Ent(P_\delta\mid\pi),
 \qquad t\ge\delta.
\]
Since $\delta\le\tau$, \eqref{eq:bridge-rate-initial-regularization}
gives, for $t\ge\tau$,
\[
 \Ent(P_t\mid\pi)
 \le C_*C_\delta e^{-r(t-\delta)}\Wass{2}(P_0,\pi)^2
 \le C_*C_\delta e^{-r(t-\tau)}\Wass{2}(P_0,\pi)^2.
\]
This proves \eqref{eq:main-wasserstein-bound}, with
$C_\tau:=C_*C_\delta$.
All estimates start at a positive time, so no finite-entropy assumption
on $P_0$ is required.
\end{proof}

\section*{AI statement}
The authors used AI tools to assist with language editing, code development, and checking mathematical arguments and proofs. The authors take full responsibility for the accuracy and integrity of the paper, including all mathematical results, proofs, and computational work.

\bibliographystyle{abbrv}
\bibliography{ref_11.7}

\end{document}